\documentclass[11pt, english, final]{article}

\usepackage[english]{babel}    
\usepackage[utf8]{inputenc}    
\usepackage[T1]{fontenc}       

\usepackage{amsmath}           

\usepackage{amssymb}          
\usepackage[amsmath,thmmarks,hyperref]{ntheorem} 
\usepackage{stmaryrd}

\usepackage[sort]{cite}        
\usepackage{setspace}          
\usepackage{exscale}           
\usepackage{xspace}            
\usepackage{bbm}               
\usepackage{mathtools}         
\usepackage{mathrsfs}          

\usepackage[expansion=false    
           ]{microtype}        
\usepackage[nottoc]{tocbibind} 
\usepackage[final=true,        
           pdfpagelabels       
           ]{hyperref}         
\usepackage[a4paper]{geometry} 
\usepackage[shortcuts]{extdash}
\usepackage{paralist}

\allowdisplaybreaks

\numberwithin{equation}{section}

\makeatletter
\g@addto@macro\bfseries{\boldmath}
\makeatother

\newcommand{\NN}{\mathbbm{N}}

\newcommand{\RR}{\mathbbm{R}}
\newcommand{\CC}{\mathbbm{C}}
\newcommand{\PP}{\mathbbm{P}}
\newcommand{\Unit}{\mathbbm 1}

\newcommand{\E}{\mathrm{e}}
\newcommand{\I}{\mathrm{i}}

\newcommand{\cc}[1]{\overline{#1}}
\newcommand{\argument}{\,\cdot\,}
\newcommand{\at}[2][]{#1|_{#2}}
\newcommand{\av}{\mathrm{av}}
\DeclareMathOperator{\supp}{supp}

\newcommand{\group}[1]{\mathrm{#1}}
\newcommand{\lie}[1]{\mathfrak{#1}}
\newcommand{\momentmap}{\mathcal{J}}
\newcommand{\racts}{\mathbin{\triangleleft}}
\newcommand{\red}{\mathrm{red}}
\newcommand{\States}{\mathcal{S}}

\newcommand{\algebra}[1]{\mathcal{#1}}
\newcommand{\A}{\algebra{A}}

\newcommand{\Q}{\algebra{Q}}
\newcommand{\Hermitian}{\mathrm{H}}
\newcommand{\Dom}{\mathcal{D}}
\newcommand{\Adbar}{\mathcal{L}^*}
\newcommand{\D}{\mathrm d}
\newcommand{\falling}[3][]{#1(#2#1)_{\downarrow,#3}}
\newcommand{\integral}[3]{\int_{#1} #2 \,#3}
\newcommand{\IsomSign}[1]{\Theta_{#1}}
\newcommand{\invstate}[1][]{\omega_{\av}^{#1}}
\newcommand{\invstateSum}[1][]{\omega_{\hbar,\mu;\av}^{#1}}
\newcommand{\punkt}{\,.}
\newcommand{\komma}{\,,}

\newcommand{\Polynomials}{\mathscr{P}}

\newcommand{\RE}             {\mathsf{Re}}
\newcommand{\IM}             {\mathsf{Im}}

\newcommand{\acts}{\triangleright}

\newcommand{\Vanishing}{\mathcal{V}}
\newcommand{\R}{\mathcal{R}}
\newcommand{\PolyHol}{\Polynomials_{\!\holomorphic,\hbar}}

\DeclarePairedDelimiter{\ordinarySet}{\{}{\}}
\DeclarePairedDelimiter{\ordinaryIP}{\langle}{\rangle}

\newcommand{\set}[3][]{\ordinarySet[#1]{\,#2 \;#1|\; #3\,}}
\newcommand{\poi}[3][]{\ordinarySet[#1]{\,#2\mathbin{,}#3\,}}
\newcommand{\dupr}[3][]{\ordinaryIP[#1]{\,#2 \,,\, #3\,}}
\newcommand{\skal}[3][]{\ordinaryIP[#1]{\,#2 \,#1|\, #3\,}}
\newcommand{\skalhbar}[3][]{\ordinaryIP[#1]{\,#2 \,#1|\, #3\,}_\hbar}
\newcommand{\skalhbarIsOne}[3][]{\ordinaryIP[#1]{\,#2 \,#1|\, #3\,}_1}

\newcommand{\genSId}[2][]{#1\langle\!#1\langle\,#2\,#1\rangle\!#1\rangle_{\ast\mathrm{id}}}
\newcommand{\abs}[2][]{#1|#2#1|}
\newcommand{\mred}{\textup{-}\mathrm{red}}
\newcommand{\holomorphic}{\mathcal{O}}
\newcommand{\Levelset}{\mathcal{Z}}

\theoremheaderfont{\normalfont\bfseries}
\theorembodyfont{\itshape}
\newtheorem{lemma}{Lemma}[section]
\newtheorem{proposition}[lemma]{Proposition}
\newtheorem{theorem}[lemma]{Theorem}
\renewtheorem*{theorem*}{Theorem}
\newtheorem{corollary}[lemma]{Corollary}
\newtheorem{definition}[lemma]{Definition}

\theoremheaderfont{\normalfont\bfseries}
\theorembodyfont{\normalfont}

\newtheorem{remark}[lemma]{Remark}

\theorembodyfont{\normalfont}
\makeatletter
\def\thmhead@plain#1#2#3{%
	\thmname{#1}\thmnumber{\@ifnotempty{#1}{ }\@upn{#2}}%
	\thmnote{ {\the\thm@notefont#3}}}
\let\thmhead\thmhead@plain
\makeatother

\theoremheaderfont{\scshape}
\theorembodyfont{\normalfont}
\theoremstyle{nonumberplain}
\theoremseparator{:}
\theoremsymbol{\hbox{$\boxempty$}}
\newtheorem{proof}{Proof}
\theoremstyle{empty}

\author{
  \\
  \textbf{Philipp Schmitt}%
  	\thanks{Supported partly by the Danish National Research Foundation through
  	the Centre of Symmetry and Deformation (DNRF92),
  	\href{mailto:schmitt@math.uni-hannover.de}{\texttt{schmitt@math.uni-hannover.de}}
  }
  \\
  Institut für Analysis, 
  Leibniz Universität Hannover
  \\[0.5cm]
  \textbf{Matthias Schötz}%
  \thanks{Boursier de l'ULB,
    supported by the Fonds de la Recherche Scientifique (FNRS) and the
    Fonds Wetenschappelijk Onderzoek - Vlaaderen (FWO) under EOS Project n$^\circ$30950721.
    Current address: Mathematisches Institut, Universit\"{a}t Leipzig,
    \href{mailto:Matthias.Schoetz@math.uni-leipzig.de}{\texttt{Matthias.Schoetz@math.uni-leipzig.de}}
  }
  \\
  Département de Mathématiques, 
  Université libre de Bruxelles
  \\[0.5cm]
}

\title{Symmetry Reduction of States II:\\A non-commutative Positivstellensatz for $\CC\PP^n$}

\date{January 2022}

\begin{document}
	\maketitle
	\vspace*{-.5cm}
	\begin{abstract}
		We give a non-commutative Positivstellensatz for $\CC\PP^n$:
		The (commutative) $^*$\=/algebra of polynomials on the real algebraic set $\CC\PP^n$ with the pointwise product can be realized
		by phase space reduction as the $\group{U}(1)$-invariant polynomials on $\CC^{1+n}$,
		restricted to the real $(2n+1)$-sphere inside $\CC^{1+n}$, and Schmüdgen's Positivstellensatz
		gives an algebraic description of the real-valued $\group{U}(1)$-invariant polynomials on $\CC^{1+n}$ that are
		strictly pointwise positive on the sphere. In analogy to this commutative case, we consider
		a non-commutative $^*$\=/algebra of polynomials on $\CC^{1+n}$, the Weyl algebra,
		and give an algebraic description of the real-valued $\group{U}(1)$-invariant polynomials
		that are positive in certain $^*$\=/representations on Hilbert spaces of holomorphic sections
		of line bundles over $\CC\PP^n$. It is especially noteworthy that the non-commutative result
		applies not only to strictly positive, but to all positive (semidefinite) elements.
		As an application, all $^*$\=/representations of the quantization of the 
		polynomial $^*$\=/algebra on $\CC\PP^n$, obtained e.g.\ through phase space reduction or Berezin--Toeplitz quantization,
		are determined.
	\end{abstract}
	\tableofcontents


\begin{onehalfspace}

\section{Introduction}
The symplectic manifolds $\CC\PP^n \cong \Levelset_\mu / \group{U}(1)$ with their usual Fubini--Study form (up to $\mu$-dependent rescalation)
arise naturally by Marsden--Weinstein reduction \cite{marsden.weinstein:reductionOfSymplecticManifoldsWithSymmetry}
from $\CC^{1+n}$ with its standard symplectic structure: One takes a $\mu$-levelset
$\Levelset_\mu \coloneqq \set{w\in \CC^{1+n}}{\momentmap(w) = \mu}$, $\mu\in {]0,\infty[}$, of the polynomial ``momentum map'' $\momentmap \coloneqq z_0 \cc{z}_0 + \dots + z_n \cc{z}_n$,
and divides out the $\group{U}(1)$-action by multiplication, which is the one generated by $\momentmap$ with respect to the standard symplectic structure of $\CC^{1+n}$.

Dual to this geometric approach, the algebra of polynomial functions on $\CC\PP^n$ (seen as a real algebraic set) can be obtained as the quotient of the
$\group{U}(1)$-invariant polynomials on $\CC^{1+n}$ modulo the ideal generated by $\momentmap - \mu \Unit$. This point of view has
the advantage that it allows a generalization to non-commutative deformations of the pointwise product. In the commutative case, Schmüdgen's Positivstellensatz \cite{schmuedgen:KMomentProblemForCompactSemiAlgebraicSets}
applies to the compact real algebraic set $\CC\PP^n$ and states that every polynomial on $\CC\PP^n$ that is pointwise
strictly positive can be expressed as a sum of squares of polynomials on $\CC\PP^n$. A slight reformulation of this result might fit somewhat
better to the setting of phase space reduction:

\begin{theorem*}[Commutative strict Positivstellensatz for $\CC\PP^n$]
  Let $f$ be a $\group{U}(1)$-invariant,\linebreak real-valued polynomial function on $\CC^{1+n}$ and $\mu\in {]0,\infty[}$.
  If $f(w)>0$ holds for all $w\in \Levelset_\mu$, then $f$ can be expressed as a sum of (Hermitian) squares of $\group{U}(1)$-invariant
  polynomials on $\CC^{1+n}$ plus an element from the ideal generated by $\momentmap - \mu \Unit$.
\end{theorem*}
Results of this type are well-known in commutative real algebraic geometry in many different settings. Most famously, Artin's solution
of Hilbert's 17th problem and the Positivstellensatz of Krivine and Stengle give an algebraic characterization of pointwise positive polynomials.
Similar theorems for non-commutative cases are less well-understood, but have been developed in a variety of different contexts:
\cite{helton:positiveNoncommutativePolynomialsAreSumsOfSquares, helton.mcCullough.putinar:nonCommutativePositivstellensatzOnIsometries, zalar:operatorPostivistellensaetzeForNoncommutativePolynomials}
discuss non-commutative polynomials, \cite{hillar.nie:elementaryAndConstructiveSolutionToHilberts17thProblemForMatrices, procesi.schacher:nonCommutativeRealNullstellensatzAndHilberts17thProblem, cimpric:realAlgebraicGeometryForMatricesOverCommutativeRings,
klep.schwaighfer:pureStatesPositiveMatrixPolynomialsAndSumsOfHermitianSquares} matrices over polynomials,
and \cite{schmuedgen:StrictPositivstellensatzForWeylAlgebra, schmuedgen:StrictPositivstellensatzForEnvelopingAlgebras, schmuedgen:algebrasOfFractionsAndStrictPositivstellensaetze}
non-commutative complex $^*$\=/algebras. See also~\cite{schmuedgen:nonCommutativeRealAlgebraicGeometry} for an overview and some more references.

The main result of the present article is a generalization of the commutative Positivstellensatz above to the deformation of $\CC\PP^n$
that is given by its Wick star product: Instead of the polynomial $^*$\=/algebra on $\CC^{1+n}$ with the pointwise product,
consider the $^*$\=/algebra of polynomials on $\CC^{1+n}$ with the Wick star product $\star_\hbar$, which is isomorphic to the Weyl algebra
of canonical commutation relations. By restricting to $\group{U}(1)$-invariant polynomials and dividing out the ideal
generated by $\momentmap-\mu\Unit$ (with respect to $\star_\hbar$), one obtains the polynomial functions on $\CC\PP^n$
with the standard Wick star product of $\CC\PP^n$ as in \cite{bordemann.brischle.emmrich.waldmann:PhaseSpaceReductionForStarProducts.ExplicitConstruction, bordemann.brischle.emmrich.waldmann:SubalgebrasWithConvergingStarProducts}. This works especially for almost all $\hbar \in {]0,\infty[}$.
The natural order on these $^*$\=/algebras associated
to $\CC^{1+n}$ and $\CC\PP^n$ is the operator order obtained by representing their elements as operators on the Fock space
or on the $\mu$-eigenspace of $\momentmap$ therein, respectively. A Positivstellensatz for the former was given
in \cite{schmuedgen:StrictPositivstellensatzForWeylAlgebra}. In the following, a similar result for the latter will be proven:

\begin{theorem*}[Non-commutative non-strict Positivstellensatz for $\CC\PP^n$]
  Let $f$ be a $\group{U}(1)$-in-\linebreak variant, real-valued polynomial function on $\CC^{1+n}$ and $\mu\in {[0,\infty[}$.
  If $\skal{\psi}{\pi_\hbar(f)(\psi)} \ge 0$ holds for all $\mu$-eigenvectors $\psi$ of $\pi_\hbar(\momentmap)$,
  where $\pi_\hbar$ denotes the representation on the Fock space, then $f$ can be expressed as a sum of Hermitian squares
  (with respect to $\star_\hbar$) of $\group{U}(1)$-invariant polynomials on $\CC^{1+n}$ plus an element
  from the ideal generated by $\momentmap - \mu \Unit$ (with respect to $\star_\hbar$).
\end{theorem*}
It is especially noteworthy that this non-commutative result appears to be stronger than expected from the analogous commutative one,
because it yields a representation as sums of Hermitian squares not only for strictly positive elements, but for all positive ones. 

Note
that $\pi_\hbar(\momentmap)$ has a discrete set of eigenvalues $\set{\hbar k}{k\in \NN_0}$, $\hbar \in {]0,\infty[}$. For all 
$\mu \in {[0,\infty[} \setminus \set{\hbar k}{k\in \NN_0}$, the above non-strict Positivstellensatz is equivalent to giving a
representation of $-\lambda \Unit$ for one arbitrary $\lambda \in {]0,\infty[}$ as a sum of Hermitian squares plus an element from the ideal generated by $\momentmap - \mu \Unit$.
In this case, a hypothetical strict Positivstellensatz, giving an algebraic certificate of positivity for $-\lambda \Unit + \epsilon \Unit$
for all $\epsilon \in {]0,\infty[}$, would trivially also give rise to a non-strict one. For $\mu  \in \set{\hbar k}{k\in \NN_0}$, however,
the appearance of a simple non-strict Positivstellensatz might be more surprising.

This article is organized as follows: After recapitulating the necessary preliminaries on ordered $^*$\=/algebras and quadratic modules in Section~\ref{sec:preliminaries},
Section~\ref{sec:reductionGeneral} is devoted to the application of the general reduction procedure for ``representable Poisson $^*$\=/algebras''
from \cite{schmitt.schoetz:preprintSymmetryReductionOfStatesI} to the case of (non-commutative) $^*$\=/algebras equipped with a Poisson bracket
coming from the commutator and equipped with an order obtained from a $^*$-representation on a pre-Hilbert space. A special case of this is the
reduction of the Wick star product from $\CC^{1+n}$ to $\CC\PP^n$ that is covered in Section~\ref{sec:reductionWick}. The proof of the main
Theorem~\ref{theorem:main} is given in Section~\ref{sec:proof}. Finally, in Section~\ref{sec:application}, this result is applied in order to determine
the $^*$-representations of the Wick star product on $\CC\PP^n$ for both strictly positive and strictly negative values of $\hbar$.

\section{Preliminaries} \label{sec:preliminaries}
The set of natural numbers is denoted by $\NN \coloneqq \{1,2,3,\dots\}$ and $\NN_0 \coloneqq \{0\} \cup \NN$. The fields of real and complex numbers are $\RR$ and $\CC$, respectively.

A \emph{$^*$-algebra} $\A$ is a unital associative algebra over $\CC$, equipped with an antilinear involution
$\argument^* \colon \A \to \A$ that fulfils $(ab)^* = b^* a^*$ for all $a,b\in \A$. The set $\A_\Hermitian \coloneqq \set{a\in \A}{a=a^*}$
of \emph{Hermitian} elements of $\A$ is a real linear subspace of $\A$. The unit of $\A$ will be denoted by $\Unit$, and always fulfils $\Unit^* = \Unit$.
A \emph{$^*$\=/ideal} of a $^*$\=/algebra $\A$ is a linear subspace $\mathcal{I}$ of $\A$ which is stable under $\argument^*$ and fulfils
$ab \in \mathcal{I}$ for all $a\in \A$, $b\in \mathcal{I}$ (hence also $b a = (a^*b^*)^* \in \mathcal{I}$).
A \emph{quadratic module} of a $^*$\=/algebra $\A$ is a subset $\Q$ of $\A_\Hermitian$ that fulfils
\begin{equation}
  q+r \in \Q
  \,,\quad\quad
  a^*q\,a \in \Q\,,
  \quad\quad\text{and}\quad\quad
  \Unit \in \Q
\end{equation}
for all $q,r\in \Q$ and all $a\in \A$. See e.g.~\cite{schmuedgen:invitationToStarAlgebras} for more details about $^*$\=/algebras
and their quadratic modules. For a quadratic module $\Q$ of $\A$ one defines the \emph{support} $\supp \Q \coloneqq \Q \cap (-\Q)$,
which is a real linear subspace of $\A_\Hermitian$ stable under conjugations $q \mapsto a^*q\,a$ with arbitrary $a\in \A$,
and the \emph{support $^*$\=/ideal} $\supp_\CC \Q \coloneqq \supp \Q + \I\, \supp \Q$,
which is a $^*$\=/ideal of $\A$. Especially if $-\Unit \in \Q$, then $\Unit \in \supp \Q \subseteq \supp_\CC \Q$ so that
$\supp_\CC \Q = \A$ and $\Q = \supp \Q = \A_\Hermitian$. Moreover, if $\Q$ and $\mathcal{I}$ are a quadratic module and a $^*$\=/ideal of $\A$,
respectively, then $\Q + \mathcal{I}_\Hermitian$ is again a quadratic module of $\A$, where $\mathcal{I}_\Hermitian \coloneqq \mathcal{I} \cap \A_\Hermitian$.
On any $^*$\=/algebra $\A$, the smallest (with respect to inclusion)
quadratic module is
\begin{equation}
  \A^{++}_\Hermitian \coloneqq \set[\Big]{\sum\nolimits_{j=1}^m a_j^* a_j}{m\in \NN; a_1,\dots,a_m \in \A}
  \komma
\end{equation}
the quadratic module of \emph{algebraically positive} Hermitian elements of $\A$, or of \emph{sums of Hermitian squares}.
If $\A$ is commutative, then sums of Hermitian squares are sums of squares because $a^*a = \RE(a)^2 + \IM(a)^2$ for all $a\in \A$.

An \emph{ordered $^*$\=/algebra} is a $^*$\=/algebra equipped with a quadratic module $\A^+_\Hermitian$ with $\supp \A^+_\Hermitian = \{0\}$.
The elements of $\A^+_\Hermitian$ will then be referred to as the \emph{positive} Hermitian elements of $\A$, and one can define a
partial order $\le$ on $\A_\Hermitian$ as $a \le b$ if and only if $b-a \in \A_\Hermitian^+$, where $a,b\in \A_\Hermitian$.
Ordered $^*$\=/algebras can be seen as generalizations of $C^*$\=/algebras that may contain unbounded elements. For example,
the basic constructions of square roots, etc., and the continuous calculus on $C^*$\=/algebras can be generalized to certain ordered $^*$\=/algebras,
see \cite{schoetz:equivalenceOrderAlgebraicStructure, schoetz:PreprintUniversalContinuousCalculusForSuStarAlgebras}.
When comparing different orderings on one $^*$\=/algebra $\A$, however, the notion of different quadratic modules on $\A$ is usually more convenient.

For example, any $^*$\=/algebra $\A$ of complex-valued functions with the usual pointwise operations and the pointwise order on the Hermitian elements
(i.e. the real-valued functions in $\A$) is an ordered $^*$\=/algebra, whose quadratic module of positive Hermitian elements $\A^+_\Hermitian$ consists
of the pointwise positive real-valued functions in $\A$. Similarly, for a pre-Hilbert space $\Dom$ with inner product
$\skal{\argument}{\argument} \colon \Dom \times \Dom \to \CC$, antilinear in the first and linear in the second argument, define $\Adbar(\Dom)$
as the ordered $^*$\=/algebra of adjointable endomorphisms of $\Dom$ with the operator order; in detail:
A linear endomorphism $a \colon \Dom \to \Dom$ is said to be \emph{adjointable} if there exists a (necessarily unique and linear) \emph{adjoint}
$a^* \colon \Dom \to \Dom$ such that $\skal{a^*(\phi)}{\psi} = \skal{\phi}{a(\psi)}$ holds for all $\phi, \psi \in \Dom$.
The set of adjointable endomorphisms of $\Dom$ is a unital subalgebra of all its linear endomorphisms and becomes a $^*$\=/algebra
when equipped with the $^*$\=/involution given by the mapping to the adjoint. It becomes an ordered $^*$\=/algebra by setting
\begin{equation}
  \Adbar(\Dom)^+_\Hermitian \coloneqq \set[\big]{a\in \Adbar(\Dom)_\Hermitian}{\skal{\psi}{a(\psi)} \ge 0 \textup{ for all $\psi\in\Dom$}}
  \punkt
\end{equation}
This yields a method for constructing quadratic modules on any $^*$\=/algebra $\A$:
Let $\pi \colon \A \to \Adbar(\Dom)$ be a $^*$\=/representation on a pre-Hilbert space $\Dom$, i.e.~a linear and multiplicative map that maps the unit of $\A$ to the unit
of $\Adbar(\Dom)$ and fulfils $\pi(a^*) = \pi(a)^*$ for all $a\in \A$, then the preimage $\Q \coloneqq \pi^{-1}\big(\Adbar(\Dom)^+_\Hermitian\big) \cap \A_\Hermitian$
is a quadratic module of $\A$ and $\supp_\CC \Q = \ker \pi$. If $\pi$ additionally is injective, then
$\Q = \pi^{-1}\big(\Adbar(\Dom)^+_\Hermitian\big)$ and $\supp_\CC \Q = \{0\}$.
The aim of this article is to give an algebraic description of some quadratic modules that are induced by $^*$\=/representations in this way.

\section{Reduction of Ordered \texorpdfstring{$^*$\=/Algebras}{*-Algebras}} \label{sec:reductionGeneral}

In \cite{schmitt.schoetz:preprintSymmetryReductionOfStatesI}, a general reduction scheme was 
developed for arbitrary ``representable Poisson $^*$\=/algebras'', which especially generalizes
Marsden--Weinstein reduction of the ordered $^*$\=/algebra of smooth functions on a symplectic manifold by the
action of a commutative Lie group. In the following, we are more interested in the case of $^*$\=/algebras represented
on a pre-Hilbert space:

A \emph{state} on an ordered $^*$\=/algebra $\A$ is a linear functional $\omega \colon \A \to \CC$ that fulfils $\dupr{\omega}{\Unit} = 1$,
$\dupr{\omega}{a} \in \RR$ for all $a\in \A_\Hermitian$, and $\dupr{\omega}{a} \in {[0,\infty[}$ for all $a\in \A^+_\Hermitian$.
We say that the order on $\A$ is \emph{induced by its states} if for all $a\in \A_\Hermitian \setminus \A^+_\Hermitian$ there exists
a state $\omega$ on $\A$ such that $\dupr{\omega}{a} < 0$. Using GNS-representations one finds that this condition is equivalent
to the existence of a pre-Hilbert space $\Dom$ and an injective $^*$-representation $\pi$ of $\A$ such that $\A^+_\Hermitian = \pi^{-1}\big(\Adbar(\Dom)^+_\Hermitian\big)$.

An \emph{eigenstate} on $\A$ of some element $a\in \A$ with \emph{eigenvalue} $\lambda \in \CC$ is a state $\omega$ on $\A$ that fulfils
one (hence all) of the following equivalent conditions, \cite[Sec.~2.3]{schmitt.schoetz:preprintSymmetryReductionOfStatesI}:
\begin{enumerate}
  \item $\dupr{\omega}{(a-\lambda\Unit)^*(a-\lambda \Unit)} = 0$.
  \item $\dupr{\omega}{a^*b} = \cc{\lambda}\dupr{\omega}{b}$ for all $b\in \A$.
  \item $\dupr{\omega}{ba} = \lambda \dupr{\omega}{b}$ for all $b\in \A$.
\end{enumerate}
The set of all eigenstates on $\A$ of $a$ with eigenvalue $\lambda$ will be denoted by $\States_{a,\lambda}(\A)$.

Now assume that $\A$ is an ordered $^*$\=/algebra whose order is induced by its states and which is equipped with the canonical
Poisson bracket given by the rescaled commutator
\begin{equation}
\poi{a}{b} \coloneqq \frac{ab-ba}{\I\hbar}
\label{eq:poissonstd}
\end{equation}
for some $\hbar \in \RR \setminus \{0\}$.
In this case, i.e.~if the Poisson bracket is of the form \eqref{eq:poissonstd}, there is no need to explicitly discuss Poisson brackets any further.
For the purpose of this article, it will also be sufficient to only consider
the case of a reduction with respect to a $1$-dimensional Lie algebra $\lie u_1 \cong \RR$, so that any (linear) ``momentum map''
$\momentmap \colon \lie u_1 \to \A_\Hermitian$ is fully determined by one ``momentum operator'' $\momentmap(1)$,
and any ``momentum'' $\mu \in \lie u_1^*$ by $\mu(1)$. By abuse of notation, we will simply write 
$\momentmap \coloneqq \momentmap(1)\in \A_\Hermitian$ and $\mu \coloneqq \mu(1) \in \RR$ as in the 
examples discussed
in \cite[Sec.~5-6]{schmitt.schoetz:preprintSymmetryReductionOfStatesI}.

In this special case, the reduction scheme of \cite[Sec.~3]{schmitt.schoetz:preprintSymmetryReductionOfStatesI} reduces to the following:
Let $\momentmap \in \A_\Hermitian$ and $\mu \in \RR$ be given. Denote by
\begin{equation}
  \A^{\lie u_1} \coloneqq \set{a\in \A}{a \momentmap = \momentmap a}
\end{equation}
the space of ``invariant'' elements, which is a unital $^*$\=/subalgebra of $\A$, i.e.~a subalgebra containing $\Unit$ and stable under $\argument^*$,
and which becomes an ordered $^*$\=/algebra by defining its quadratic module of positive Hermitian elements to be 
$(\A^{\lie u_1})^+_\Hermitian \coloneqq \A^+_\Hermitian \cap \A^{\lie u_1}$.
Moreover, denote by
\begin{align}
  \R_\mu &\coloneqq \set[\big]{a \in (\A^{\lie u_1})_\Hermitian}{ \dupr{\omega}{a} \ge 0 \textup{ for all 
  }\omega \in \States_{\momentmap,\mu}(\A^{\lie u_1}) }
  \shortintertext{and}
  \Vanishing_\mu &\coloneqq \set[\big]{a \in \A^{\lie u_1}}{ \dupr{\omega}{a} = 0 \textup{ for all 
  }\omega \in \States_{\momentmap,\mu}(\A^{\lie u_1}) } = \supp_\CC \R_\mu
\end{align}
the \emph{quadratic module $\R_\mu$ of the $\momentmap$-reduction of $\A$ at $\mu$}
and the \emph{vanishing ideal} $\Vanishing_\mu$ of the eigenstates of $\momentmap$ with eigenvalue $\mu$.
Note that $\R_\mu$ and $\Vanishing_\mu$ are a quadratic module and a $^*$\=/ideal, respectively, of $\A^{\lie u_1}$,
but not of $\A$ in general, and that $(\A^{\lie u_1})^+_\Hermitian \subseteq \R_\mu$ and $\momentmap-\mu \Unit \in \Vanishing_\mu$ hold.
In particular, $\genSId{\momentmap-\mu}$, the $^*$\=/ideal of $\A^{\lie u_1}$ that is generated by $\momentmap-\mu \Unit$, is contained in $\Vanishing_\mu$.

The \emph{$\momentmap$-reduction of $\A$ at $\mu$} is then given by the tuple
of an ordered $^*$\=/algebra $\A_{\mu\mred}$ and a map $[\argument]_\mu \colon \A^{\lie u_1} \to \A_{\mu\mred}$
that can be constructed as follows:
\begin{itemize}
  \item The $^*$\=/algebra underlying $\A_{\mu\mred}$ is the quotient $^*$\=/algebra $\A^{\lie u_1} / \Vanishing_\mu$.
  \item The map $[\argument]_\mu$ is the canonical projection onto the quotient.
  \item The quadratic module of positive Hermitian elements of $\A_{\mu\mred}$ is $(\A_{\mu\mred})^+_\Hermitian \coloneqq \set{[r]_\mu}{r\in\R_\mu}$.
\end{itemize}
This construction of $\A_{\mu\mred}$ and $[\argument]_\mu \colon \A^{\lie u_1} \to \A_{\mu\mred}$ is an application of 
\cite[Sec.~3.3]{schmitt.schoetz:preprintSymmetryReductionOfStatesI}; for the sake of the present
article we can equally well view this as a definition. 
The order on $\A_{\mu\mred}$ is again induced by its states, because every $\omega \in \States_{\momentmap,\mu}(\A^{\lie u_1})$
descends to a state on $\A_{\mu\mred}$ and because the Hermitian elements in the preimage of $(\A_{\mu\mred})^+_\Hermitian$ under $[\argument]_\mu$ are $\R_\mu$.
\begin{proposition} \label{proposition:Rmucharacterization}
  Let $\A$ be an ordered $^*$\=/algebra, $\momentmap \in \A_\Hermitian$, $\Dom$ a pre-Hilbert space and $\pi \colon \A \to \Adbar(\Dom)$ an injective $^*$\=/representation
  of the $^*$\=/algebra underlying $\A$ such that $\A^+_\Hermitian = \pi^{-1}\big( \Adbar(\Dom)^+_\Hermitian\big)$.
  For any $\mu \in \RR$ write
  \begin{equation}
    \mathcal{E}_\mu \coloneqq \set[\big]{\psi \in \Dom}{\pi(\momentmap)(\psi) = \mu \psi}
    \label{eq:mueigenspace}
  \end{equation}
  for the $\mu$-eigenspace of $\pi(\momentmap)$, and 
  $\mathcal{E}_\mu^\bot \coloneqq \set{\phi\in\Dom}{\skal{\psi}{\phi} = 0 \textup{ for all }\psi \in \mathcal{E}_\mu}$
  for its orthogonal complement in $\Dom$. 
  Let $\mu \in \RR$ be given and assume that $\Dom = \mathcal{E}_\mu \oplus \mathcal{E}_\mu^\bot$ as vector spaces
  and that there exists $\epsilon \in {]0,\infty[}$ such that $\skal[\big]{\phi}{\pi\big((\momentmap-\mu\Unit)^2\big)(\phi)} \ge \epsilon \skal{\phi}{\phi}$
  holds for all $\phi \in \mathcal{E}_\mu^\bot$. Then 
  \begin{equation}
    (\A^{\lie u_1})^+_\Hermitian + \big( \genSId{\momentmap-\mu} \big)_\Hermitian
    =
    \R_{\mu}
    =
    \set[\big]{a\in (\A^{\lie u_1})_\Hermitian}{\skal{\psi}{\pi(a)(\psi)}\ge 0 \textup{ for all }\psi 
    \in \mathcal{E}_\mu }
    \komma
    \label{eq:Rmucharacterization}
  \end{equation}
  and the $^*$\=/algebra underlying $\A_{\mu\mred}$ admits an injective $^*$\=/representation $\pi_{\mu\mred} \colon \A_{\mu\mred} \to \Adbar(\mathcal{E}_\mu)$
  such that $(\A_{\mu\mred})^+_\Hermitian = \pi_{\mu\mred}^{-1}\big( \Adbar(\mathcal{E}_\mu)^+_\Hermitian\big)$.
  This $^*$\=/representation is given by $\pi_{\mu\mred}([a]_\mu)(\psi) \coloneqq \pi(a)(\psi) \in \mathcal{E}_\mu$ for all $\psi \in \mathcal{E}_\mu$ and all $a\in \A^{\lie u_1}$.
\end{proposition}
\begin{proof}
	The inclusion $(\A^{\lie u_1})^+_\Hermitian + ( \genSId{\momentmap-\mu} )_\Hermitian \subseteq \R_{\mu}$
	follows immediately from the properties of eigenstates of $\momentmap$ with eigenvalue $\mu$ and the fact that $\momentmap$ is central in $\A^{\lie u_1}$.
	For every $\psi \in \mathcal{E}_\mu$ with $\skal{\psi}{\psi} = 1$, the map $\chi_\psi \colon \A^{\lie u_1} \to \CC$,
	$a \mapsto \dupr{\chi_\psi}{a} \coloneqq \skal{\psi}{\pi(a)(\psi)}$
	is a state on $\A^{\lie u_1}$ and even is an eigenstate of $\momentmap$ with eigenvalue $\mu$.
	So 
	$\R_{\mu} \subseteq \set{a\in (\A^{\lie u_1})_\Hermitian}{\skal{\psi}{\pi(a)(\psi)}\ge 0 \textup{ for all }\psi \in \mathcal{E}_\mu }$.
	
	In order to prove \eqref{eq:Rmucharacterization} it remains to show that the right-hand side is contained 
	in the left-hand side of this equation.
	Note that $\pi(a)(\mathcal{E}_\mu) \subseteq \mathcal E_\mu$ for all $a\in \A^{\lie u_1}$,
	because $\pi(\momentmap)(\pi(a)(\psi)) = \pi(a)(\pi(\momentmap)(\psi)) = \mu \,\pi(a)(\psi)$ holds for all $\psi \in \mathcal E_\mu$,
	and also $\pi(a)(\mathcal{E}^\bot_\mu) \subseteq \mathcal E_\mu^\bot$, 
	because $\skal{\psi}{\pi(a)(\phi)} = \skal{\pi(a^*) (\psi)}{\phi} = 0$ holds for all $\psi \in \mathcal E_\mu$ and $\phi \in \mathcal E_\mu^\bot$.
	Now assume that $a\in (\A^{\lie u_1})_\Hermitian$ fulfils $\skal{\psi}{\pi(a)(\psi)}\ge 0$ for all $\psi \in \mathcal{E}_\mu$.
	Then $\hat{a} \coloneqq a + (2\epsilon)^{-1}(\momentmap-\mu\Unit)^2(a-\Unit)^2$ is Hermitian and an element of $\A^{\lie u_1}$
	because $\momentmap a = a \momentmap$,
	and even $\hat{a} \in (\A^{\lie u_1})^+_\Hermitian$ holds: Indeed, by assumption,
	any $\phi \in \Dom$ can be decomposed as $\phi = \phi_\mu + \phi_\bot$ with $\phi_\mu \in 
	\mathcal{E}_\mu$ and $\phi_\bot \in \mathcal{E}_\mu^\bot$, and thus
	\begin{align*}
		\skal{\phi}{\pi(\hat{a})(\phi)} 
		&=
		\skal{\phi_\mu}{\pi(\hat{a})(\phi_\mu)}
		+
		\skal{\phi_\bot}{\pi(\hat{a})(\phi_\bot)}
		\\
		&=
		\underbrace{\skal{\phi_\mu}{\pi(a)(\phi_\mu)}}_{\ge 0}{}
		+
		\skal{\phi_\bot}{\pi(a)(\phi_\bot)}
		+
		\frac{1}{2\epsilon}
		\underbrace{
			\skal[\big]{\pi(a-\Unit)(\phi_\bot)}{\pi\big((\momentmap-\mu\Unit)^2(a-\Unit)\big)(\phi_\bot)}  
		}_{
			\ge \epsilon \skal{\pi(a-\Unit)(\phi_\bot)}{\pi(a-\Unit)(\phi_\bot)} \textup{ by assumption}
		}
		\\  
		&\ge
		\skal[\big]{\phi_\bot}{\pi\big(a+(a-\Unit)^2/2 \big)(\phi_\bot)}
		\\
		&=
		\skal[\big]{\phi_\bot}{\pi\big((a^2 + \Unit)/2\big)(\phi_\bot)}
		\\
		&\ge
		0 \komma
	\end{align*}
	showing that $\hat{a} \in \pi^{-1}\big( \Adbar(\Dom)^+_\Hermitian \big) = \A_\Hermitian^+$. It follows that $\hat{a} \in (\A^{\lie u_1})^+_\Hermitian$,
	hence $a \in (\A^{\lie u_1})^+_\Hermitian + ( \genSId{\momentmap-\mu} )_\Hermitian$.
	
	Since $\pi(a)(\mathcal E_\mu) \subseteq \mathcal E_\mu$ for all $a \in \A^{\lie u_1}$
	and since $\chi_\psi$ is an eigenstate of $\momentmap$ with eigenvalue $\mu$ for all normalized $\psi \in \mathcal E_\mu$,
	it follows that $\pi_{\mu\mred}$ is a well-defined $^*$\=/representation of $\A_{\mu\mred}$,
	and that $(\A_{\mu\mred})^+_\Hermitian \subseteq \pi_{\mu\mred}^{-1}\big( \Adbar(\mathcal{E}_\mu)^+_\Hermitian \big)$.
	Using \eqref{eq:Rmucharacterization} one can check that $\pi_{\mu\mred}$ is injective.
	In order to show that $(\A_{\mu\mred})^+_\Hermitian \supseteq \pi_{\mu\mred}^{-1}\big( \Adbar(\mathcal{E}_\mu)^+_\Hermitian \big)$,
	let $[a]_\mu \in \pi_{\mu\mred}^{-1}\big( \Adbar(\mathcal{E}_\mu)^+_\Hermitian \big)$ be given.
	Since $\pi_{\mu\mred}$ is injective, $[a]_\mu$ is necessarily Hermitian, and therefore has a representative $a \in (\A^{\lie u_1})_\Hermitian$
	(which e.g.~can be constructed as the Hermitian part of any representative). 
	From $\skal{\psi}{\pi(a)(\psi)} = \skal{\psi}{\pi_{\mu\mred}([a]_\mu)(\psi)} \ge 0$ for all $\psi \in \mathcal E_\mu$
	it follows that $a \in \R_\mu$, so $[a]_\mu \in (\A_{\mu\mred})^+_\Hermitian$.
\end{proof}
If Proposition~\ref{proposition:Rmucharacterization} applies, then the general reduction scheme from \cite{schmitt.schoetz:preprintSymmetryReductionOfStatesI}
yields the naively expected result for ordered $^*$\=/algebras of operators. This is completely analogous to
the reduction of Poisson manifolds discussed in \cite[Sec.~4]{schmitt.schoetz:preprintSymmetryReductionOfStatesI}, with evaluation
functionals at points of the $\mu$-levelset being replaced by vector states of $\mu$-eigenvectors.

However, the assumptions of Proposition~\ref{proposition:Rmucharacterization} are not fulfilled in
all ``non-pathological'' cases: An instructive example with different behaviour is the reduction of the Weyl algebra with
respect to translation symmetry that has been examined in \cite[Sec.~5]{schmitt.schoetz:preprintSymmetryReductionOfStatesI}.
There, the momentum operator has continuous spectrum and no eigenvectors. In contrast to this, one finds for momentum operators with discrete spectrum:
\begin{proposition} \label{proposition:discreteSpectrumGood}
  Let $\A$ be an ordered $^*$\=/algebra, $\momentmap \in \A_\Hermitian$, $\Dom$ a pre-Hilbert space and $\pi \colon \A \to \Adbar(\Dom)$ an injective $^*$\=/representation
  of the $^*$\=/algebra underlying $\A$ such that $\A^+_\Hermitian = \pi^{-1}\big( \Adbar(\Dom)^+_\Hermitian\big)$.
  Denote again the $\mu$-eigenspace of $\pi(\momentmap)$ by $\mathcal{E}_\mu$ like in \eqref{eq:mueigenspace}.
  Moreover, assume that the set $\set{\mu \in \RR}{\mathcal{E}_\mu \neq \{0\}}$ of eigenvalues of $\pi(\momentmap)$
  is discrete and that $\Dom = \bigoplus_{\mu \in \RR} \mathcal{E}_\mu$ as vector spaces. 
  Then the assumptions of the previous Proposition~\ref{proposition:Rmucharacterization}
  are fulfilled for all $\mu \in \RR$, i.e.~the decomposition $\Dom = \mathcal{E}_\mu \oplus \mathcal{E}_\mu^\bot$
  holds and there exists $\epsilon \in {]0,\infty[}$ such that 
  $\skal[\big]{\phi}{\pi\big((\momentmap-\mu\Unit)^2\big)(\phi)} \ge \epsilon \skal{\phi}{\phi}$
  for all $\phi \in \mathcal{E}_\mu^\bot$.
\end{proposition}
\begin{proof}
  Let $\mu \in \RR$ be given. Then $\mathcal{E}_\mu^\bot = \bigoplus_{\mu' \in \RR \setminus \{\mu\}} \mathcal{E}_{\mu'}$
  because eigenvectors of $\pi(\momentmap)$ to different eigenvalues are orthogonal, and so $\Dom = \mathcal{E}_\mu \oplus \mathcal{E}_\mu^\bot$.
  If $\mathcal{E}_\mu^\bot = \{0\}$, then there is nothing else to show. Otherwise, let
  \begin{equation*}
    \epsilon \coloneqq \min \set[\big]{ (\mu-\mu')^2 }{ \mu' \in \RR \setminus\{\mu\} \textup{ such that }\mathcal{E}_{\mu'} \neq \{0\}} \in {]0,\infty[}
    \komma
  \end{equation*}
  then $\skal[\big]{\phi}{\pi\big((\momentmap-\mu\Unit)^2\big)(\phi)} \ge \epsilon \skal{\phi}{\phi}$ for all 
  $\phi \in \mathcal{E}_{\mu'}$ with $\mu' \in \RR \setminus\{\mu\}$, hence even for all $\phi \in \mathcal{E}_\mu^\bot$.
\end{proof}
\begin{remark} \label{remark:totalspacequantization}
  These assumptions hold in at least one important class of examples:
  In \cite{bordemann.roemer:totalSpaceQuantizationOfKaehlerManifolds}, the ``total space quantization'' of compact Kähler manifolds was discussed, essentially
  quantizing a $^*$\=/algebra of functions on a holomorphic line bundle over a compact Kähler manifold, which results in a
  construction carrying the same information as the usual quantization of compact Kähler manifolds. The relation between this
  ``total space quantization'' of compact Kähler manifolds and the usual one that quantizes a $^*$\=/algebra of functions on the Kähler manifold itself,
  is essentially given by the above reduction scheme, with a representation space decomposing as a direct sum of eigenspaces to a discrete set of eigenvalues.
  The quantization of $\CC\PP^n$ that will be discussed in the following is one typical example thereof, its ``total space quantization'' is given by
  the Wick star product on $\CC^{1+n}$.
\end{remark}

\section{Reduction of the Wick Star Product on \texorpdfstring{$\CC^{1+n}$}{C(1+n)}} \label{sec:reductionWick}

Fix a number $n\in \NN$ for the rest of this article.
As an example of the reduction scheme for ordered $^*$\=/algebras from the last section
we consider the Wick star product, which describes the Weyl algebra of canonical commutation relations:

Denote by $z_0,\dots,z_n, \cc{z_0}, \dots, \cc{z_n} \colon \CC^{1+n} \to \CC$ the standard complex coordinates on $\CC^{1+n}$ and their
complex conjugates, and write
\begin{equation}
  z^K \coloneqq (z_0)^{K_0} \dots (z_n)^{K_n}\quad\quad\text{and} \quad\quad \cc{z}^K \coloneqq (\cc{z_0})^{K_0} \dots (\cc{z_n})^{K_n}
\end{equation}
with multiindex $K\in \NN_0^{1+n}$ for the holomorphic and antiholomorphic monomials, respectively.
We will use standard multiindex notation, especially $\abs{K} \coloneqq K_0 + \dots + K_n$ and $K! \coloneqq K_0! \dots K_n!$.
Moreover, for all $k,\ell \in \NN_0$, let $\Polynomials^{k,\ell}(\CC^{1+n})$ be the $\CC$-linear span of the monomials $z^K \cc{z}^L$ with
$K,L \in \NN_0^{1+n}$ fulfilling $\abs{K} = k$ and $\abs{L} = \ell$, and let
\begin{equation}
  \Polynomials(\CC^{1+n}) \coloneqq \bigoplus_{k,\ell \in \NN_0} \Polynomials^{k,\ell}(\CC^{1+n})
  \label{eq:polynomials}
\end{equation}
be their direct sum, i.e.~the space of all $\CC$-valued (not necessarily holomorphic) polynomial functions on $\CC^{1+n}$. 
The \emph{Wick star product} is, for any $\hbar \in \RR$, the bilinear associative product $\star_\hbar$ on $\Polynomials(\CC^{1+n})$ which is defined as
\begin{equation} \label{eq:def:wickproduct}
  f \star_\hbar g
  \coloneqq
  \sum_{K\in \NN_0^{1+n}} \frac{\hbar^{\abs{K}}}{K!} \frac{\partial^{\abs{K}} f}{\partial \cc{z}^K} \frac{\partial^{\abs{K}} g}{\partial z^K}
  =
  \sum_{t=0}^\infty \frac{\hbar^t}{t!} \sum_{i_1,\dots,i_t = 0}^n \frac{\partial^t f}{\partial \cc{z}_{i_1} \dots \partial \cc{z}_{i_t}} \frac{\partial^t g}{\partial z_{i_1} \dots \partial z_{i_t}}
  \in
  \Polynomials(\CC^{1+n})
\end{equation}
for all $f,g\in \Polynomials(\CC^{1+n})$, where
\begin{equation}
  \frac{\partial^{\abs{K}}}{\partial z^K} \coloneqq \bigg( \frac{\partial}{\partial z_0}\bigg)^{K_0} \dots \bigg( \frac{\partial}{\partial z_n}\bigg)^{K_n}
  \quad\quad\text{and} \quad\quad
  \frac{\partial^{\abs{K}}}{\partial \cc{z}^K} \coloneqq \bigg( \frac{\partial}{\partial \cc{z_0}}\bigg)^{K_0} \dots \bigg( \frac{\partial}{\partial \cc{z_n}}\bigg)^{K_n}
\end{equation}
for all $K\in \NN_0^{1+n}$. One can check that the complex vector space $\Polynomials(\CC^{1+n})$, together with the Wick star product and the
$^*$\=/involution of pointwise complex conjugation, becomes a $^*$-algebra whose unit is the constant $1$-function. This $^*$\=/algebra
will be denoted by $\Polynomials_\hbar(\CC^{1+n})$. Of course, if $\hbar = 0$,
then $\star_0$ is just the pointwise product. Some elementary properties of the Wick star product are easy to check:
\begin{proposition} \label{proposition:starhbarbasics}
  For all $f \in \Polynomials^{k,\ell}(\CC^{1+n})$, $g\in \Polynomials^{r,s}(\CC^{1+n})$ and all $\hbar \in \RR$, the product $f \star_\hbar g$
  is of the form
  \begin{equation}
    f \star_\hbar g = \sum_{t=0}^{\min \{\ell, r\}} h_t
  \end{equation}
  with certain $h_t \in \Polynomials^{k+r-t,\ell+s-t}(\CC^{1+n})$, and especially $h_0 = fg$. Consequently,
  the unital $^*$\=/subalgebra of $\Polynomials_\hbar(\CC^{1+n})$ that is generated by the degree-$1$-monomials
  $z_0,\dots,z_n,\cc{z}_0,\dots,\cc{z}_n$ is whole $\Polynomials_\hbar(\CC^{1+n})$. Moreover, for all $f,g\in\Polynomials(\CC^{1+n})$,
  \begin{equation}
    \lim_{\hbar \to 0} \frac{f \star_\hbar g -g \star_\hbar f}{\I\hbar} 
    =
    \frac{1}{\I} \sum_{j=0}^n \bigg( \frac{\partial f}{\partial \cc{z_j}} \frac{\partial g}{\partial z_j} - \frac{\partial f}{\partial z_j} \frac{\partial g}{\partial \cc{z_j}}\bigg)
    =
    \poi{f}{g}
    \komma
    \label{eq:poissonlim}
  \end{equation}
  where $\poi{\argument}{\argument}$ is the Poisson bracket associated to the standard Kähler structure of $\CC^{1+n}$.
\end{proposition}
We will in the following be mainly interested in the case $\hbar > 0$, the case $\hbar < 0$ 
can be reduced to the positive one, which will be discussed in Section~\ref{sec:application}.

For $\hbar \in {]0,\infty[}$, a $^*$\=/representation of the $^*$\=/algebra $\Polynomials_\hbar(\CC^{1+n})$ can be obtained by the 
GNS-construction for the evaluation functional at $0$. This results in a $^*$\=/representation
by polynomial holomorphic differential operators on the pre-Hilbert space $\PolyHol(\CC^{1+n})$ of holomorphic polynomials on $\CC^{1+n}$,
\begin{equation}
  \PolyHol(\CC^{1+n}) \coloneqq \bigoplus_{k \in \NN_0} \Polynomials^{k,0}(\CC^{1+n})
\end{equation}
with inner product $\skalhbar{\argument}{\argument}$ defined as
\begin{align}
  \skalhbar{f}{g} &\coloneqq \frac{1}{(\hbar\pi)^{1+n}}\integral{\CC^{1+n}}{\cc{f} g \exp(- \momentmap/\hbar )}{ \D^{1+n} z \,\D^{1+n} \cc{z} }
  \label{eq:skalhbar}
  \shortintertext{with}
  \momentmap &\coloneqq \sum_{j=0}^n z_j \cc{z}_j \in \Polynomials^{1,1}(\CC^{1+n}) \label{eq:momentmap}
\end{align}
and where $\D^{1+n} z \,\D^{1+n} \cc{z}$ is the Lebesgue measure on $\CC^{1+n}$. This inner product especially fulfils
\begin{equation}
  \skalhbar{z^K}{z^L} = \delta_{K,L} \hbar^{\abs{K}} K!
\end{equation}
for all $K,L \in \NN_0^{1+n}$, where $\delta_{K,L} \coloneqq 1$ if $K=L$ and otherwise $\delta_{K,L} \coloneqq 0$. The 
pre-Hilbert space $\PolyHol(\CC^{1+n})$ (or its completion) is usually referred to as the \emph{Fock space} or \emph{Segal--Bargmann space}.
For every $\hbar \in {]0,\infty[}$, the map $\pi_\hbar \coloneqq \Polynomials_\hbar(\CC^{1+n}) \to \Adbar\big( \PolyHol(\CC^{1+n})\big)$,
$f \mapsto \pi_\hbar(f)$, defined by
\begin{equation}
  \pi_\hbar(z^K \cc{z}^L)
  \coloneqq
  z^K \hbar^{\abs{L}} \frac{\partial^{\abs{L}}}{\partial z^L}
\end{equation}
for all $K,L\in \NN_0^{1+n}$, describes an injective $^*$-representation of 
$\Polynomials_\hbar(\CC^{1+n})$ by differential operators on $\PolyHol(\CC^{1+n})$.
Via this $^*$\=/representation $\pi_\hbar$, the $^*$\=/algebra $\Polynomials_\hbar(\CC^{1+n})$
is isomorphic to the Weyl algebra of canonical commutation relations in its representation on the Fock space.
One can now pull back the quadratic module $\Adbar( \PolyHol(\CC^{1+n}))^+_\Hermitian$ via $\pi_\hbar$ to $\Polynomials_\hbar(\CC^{1+n})$,
thus turning $\Polynomials_\hbar(\CC^{1+n})$ into an ordered $^*$\=/algebra:
\begin{definition}
  For $\hbar \in {]0,\infty[}$ define the quadratic module
  $\Polynomials_\hbar(\CC^{1+n})^+_\Hermitian \coloneqq \pi^{-1}_\hbar\big( \Adbar( \PolyHol(\CC^{1+n}))^+_\Hermitian \big)$.
\end{definition}
The group $\group{U}(1+n)$ of unitary $(1+n)\times(1+n)$\,-matrices acts on $\CC^{1+n}$ by multiplication from the left,
$\argument\acts\argument \colon \group{U}(1+n) \times \CC^{1+n} \to \CC^{1+n}$, $(u,w) \mapsto u \acts w \coloneqq uw$. 
From this one obtains a right action on spaces of $(k,\ell)$-homogeneous polynomials, $k,\ell \in \NN_0$,
namely $\argument \racts \argument \colon \Polynomials^{k,\ell}(\CC^{1+n}) \times \group{U}(1+n) \to \Polynomials^{k,\ell}(\CC^{1+n})$,
$(f,u) \mapsto f \racts u$ with $(f \racts u)(w) \coloneqq f(u\acts w)$ for all $w\in \CC^{1+n}$, and consequently also actions on
$\Polynomials_\hbar(\CC^{1+n})$ and $\PolyHol(\CC^{1+n})$. Note that especially $z_i \racts u = \sum_{j=0}^n u_{i,j} z_j$ holds for all $i \in \{0,\dots,n\}$.
It follows immediately from the second identity in \eqref{eq:def:wickproduct} that $\star_\hbar$ is $\group{U}(1+n)$-equivariant, i.e.
\begin{equation}
  (f \racts u) \star_\hbar (g\racts u) = (f\star_\hbar g) \racts u
\end{equation}
for all $f,g\in \Polynomials_\hbar(\CC^{1+n})$ and all $u\in \group{U}(1+n)$. Similarly, by $\group{U}(1+n)$-invariance of \eqref{eq:skalhbar},
the group $\group{U}(1+n)$ acts unitarly on $\PolyHol(\CC^{1+n})$, i.e.
\begin{equation}
  \skalhbar{f\racts u}{g\racts u} = \skalhbar{f}{g}
\end{equation}
for all $f,g\in \PolyHol(\CC^{1+n})$ and all $u\in \group{U}(1+n)$. Moreover,
\begin{equation}
  \pi_\hbar(f\racts u)(g\racts u) = \pi_\hbar(f)(g) \racts u
\end{equation}
holds for all $f \in \Polynomials_\hbar(\CC^{1+n})$, $g\in \PolyHol(\CC^{1+n})$, and all $u \in\group{U}(1+n)$,
which can easily be checked for generators $z_0,\dots,z_n,\cc{z}_0, \dots, \cc{z}_n$ in place of $f$.
It especially follows that the action of $\group{U}(1+n)$ on $\Polynomials_\hbar(\CC^{1+n})$ preserves the quadratic module $\Polynomials_\hbar(\CC^{1+n})^+_\Hermitian$.

The action of the diagonal $\group{U}(1)$-subgroup of $\group{U}(1+n)$ is generated by the polynomial $\momentmap$ from \eqref{eq:momentmap},
which means that
\begin{equation}
  \frac{\D}{\D t}\at[\bigg]{0} \big(f \racts \E^{\I t} \Unit_{1+n}\big) = \I \deg_{\holomorphic-\cc{\holomorphic}} f = \frac{f \star_\hbar \momentmap - \momentmap \star_\hbar f}{\I \hbar}
\end{equation}
with $\Unit_{1+n} \in \group{U}(1+n)$ the $(1+n)\times(1+n)$\,-identity matrix, and where $\deg_{\holomorphic-\cc{\holomorphic}}$ is the
derivation on $\Polynomials_\hbar(\CC^{1+n})$
of holomorphic minus antiholomorphic degree, i.e.~$\deg_{\holomorphic-\cc{\holomorphic}} z^K \cc{z}^L = (\abs{K}-\abs{L}) z^K \cc{z}^L$
for all $K,L\in \NN_0^{1+n}$. The $\momentmap$-reduction of $\Polynomials_\hbar(\CC^{1+n})$ at arbitrary $\mu \in \RR$ is easy to describe
because Proposition~\ref{proposition:discreteSpectrumGood} applies to $\Polynomials_\hbar(\CC^{1+n})$ with $^*$\=/representation $\pi_\hbar$ on the Fock space:
The unital $^*$\=/subalgebra of invariant elements is
\begin{equation}
  \Polynomials_\hbar(\CC^{1+n})^{\lie u_1} = \bigoplus_{k\in \NN_0} \Polynomials^{k,k}(\CC^{1+n}) \label{eq:WeylESdecomp}
\end{equation}
and $\pi_\hbar(\momentmap) = \hbar \deg$ with $\deg \in \Adbar(\PolyHol(\CC^{1+n}))$ the degree derivation, $\deg z^K = \abs K z^K$ for all $K\in \NN_0^{1+n}$.
This especially means that $\pi_\hbar(\momentmap)$ has a discrete set of eigenvalues $\set{\hbar k}{k\in \NN_0}$ and that
$\PolyHol(\CC^{1+n})$ decomposes into a direct sum of eigenspaces $\Polynomials^{k,0}(\CC^{1+n})$, $k\in \NN_0$, of $\pi_\hbar(\momentmap)$.
Propositions~\ref{proposition:Rmucharacterization} and \ref{proposition:discreteSpectrumGood} therefore show:
\begin{corollary} \label{corollary:neu}
  For all $\hbar \in {]0,\infty[}$ and all $\mu \in \RR$, the quadratic module $\R_{\hbar,\mu}$ of the $\momentmap$-reduction of $\Polynomials_\hbar(\CC^{1+n})$ at $\mu$ is
  \begin{equation}
    \R_{\hbar,\mu}
    =
    \big(\Polynomials_\hbar(\CC^{1+n})^{\lie u_1}\big) ^+_\Hermitian + \big( \genSId{\momentmap-\mu} \big)_\Hermitian
    =
    \set[\bigg]{f \in \big(\Polynomials_\hbar(\CC^{1+n})^{\lie u_1}\big)_\Hermitian }{ \begin{array}{l}
                                                       \skalhbar{g}{\pi_\hbar(f)(g)} \ge 0\\
                                                       \textup{ for all }g \in \Polynomials^{k,0}(\CC^{1+n})
                                                     \end{array}
                                                  }
    \komma \label{eq:Rmu}
  \end{equation}
  and especially $\R_{\hbar,\mu} \neq \big(\Polynomials_\hbar(\CC^{1+n})^{\lie u_1}\big)_\Hermitian$ if and only if $\mu \in \set{\hbar k}{k\in \NN_0}$.
  Moreover, if $\mu = \hbar k$ with $k\in \NN_0$, then the codimension of $\supp_\CC \R_{\hbar,\hbar k}$ in $\Polynomials_\hbar(\CC^{1+n})^{\lie u_1}$
  is finite, for every $f\in \Polynomials_\hbar(\CC^{1+n})^{\lie u_1}$ there exists a unique $g\in \Polynomials^{k,k}(\CC^{1+n})$ such that
  $f-g \in \supp_\CC \R_{\hbar,\hbar k}$, and $\Polynomials_\hbar(\CC^{1+n})_{\hbar k\mred}$
  is isomorphic as an ordered $^*$\=/algebra to $\Adbar(\Polynomials^{k,0}(\CC^{1+n}))$
  via the reduced representation $(\pi_\hbar)_{\hbar k\mred}$ constructed in Proposition~\ref{proposition:Rmucharacterization}.
  If $\mu \in \RR \setminus \set{\hbar k}{k\in \NN_0}$, however, then $\Polynomials_\hbar(\CC^{1+n})_{\mu\mred}= \{0\}$.
\end{corollary}
\begin{proof}
  It is only left to show that for all $k\in \NN_0$, the reduced representation $(\pi_\hbar)_{\hbar k\mred}$ from Proposition~\ref{proposition:Rmucharacterization}
  is surjective, and that every $f\in \Polynomials_\hbar(\CC^{1+n})^{\lie u_1}$ coincides modulo $\supp_\CC \R_{\hbar,\hbar k}$ with a unique $g\in \Polynomials^{k,k}(\CC^{1+n})$.
  As $\Polynomials_\hbar(\CC^{1+n})_{\hbar k\mred} = \Polynomials_\hbar(\CC^{1+n})^{\lie u_1} / \supp_\CC\R_{\hbar,\hbar k}$,
  both these statements follow from the identity
  \begin{equation*}
    \skalhbar[\big]{z^K}{\pi_\hbar(z^L \cc{z}^M)(z^N)} = \delta_{K,L} \delta_{M,N} \hbar^{2k}
  \end{equation*}
  for $K,L,M,N \in \NN_0^{1+n}$ with $\abs{K} = \abs{L} = \abs{M} = \abs{N} = k$, and by counting dimensions.
\end{proof}
Note that the dimension of $\Polynomials^{k,0}(\CC^{1+n})$ is $d_{n,k} \coloneqq \binom{n+k}{k}$ 
so that $\Polynomials_\hbar(\CC^{1+n})_{\hbar k\mred}$ is isomorphic to the matrix $^*$\=/algebra $\CC^{d_{n,k} \times d_{n,k}}$
with the quadratic module of positive-semidefinite matrices.

\begin{remark}
	The reduced algebra $\Polynomials_\hbar(\CC^{1+n})_{\hbar k\mred}$ is related to the geometry of $\CC\PP^n$ by noting that the spaces $\Polynomials^{k,0}(\CC^{1+n})$, $k\in \NN_0$,
	of $k$-homogeneous holomorphic polynomials on $\CC^{1+n}$ are isomorphic to the spaces of all holomorphic sections of the $k$-th tensor power of
	the hyperplane bundle (the dual of the tautological bundle) over $\CC\PP^n$. This actually works in all examples of the type mentioned in 
	Remark~\ref{remark:totalspacequantization}, turning the reduction procedure into the inverse of the construction of the total space quantization
	of \cite{bordemann.roemer:totalSpaceQuantizationOfKaehlerManifolds}.
\end{remark}
Seeing Corollary~\ref{corollary:neu} in the general context of \cite{schmitt.schoetz:preprintSymmetryReductionOfStatesI} has the
advantage that it makes the connection to the commutative case more than a mere heuristic:
For $\hbar = 0$, consider the ordered $^*$\=/algebra $\Polynomials_0(\CC^{1+n})$ with the pointwise order, i.e.~with the
quadratic module $\Polynomials_0(\CC^{1+n})^+_\Hermitian$ of pointwise positive polynomials, and endowed with the
Poisson bracket of \eqref{eq:poissonlim}. Then the reduction scheme of \cite{schmitt.schoetz:preprintSymmetryReductionOfStatesI}
yields a completely analogous result, and one can even give a mostly algebraic description of the quadratic module
of the reduction:
\begin{proposition} \label{proposition:classicalReduction}
  \textup{(See \cite[Sec.~6]{schmitt.schoetz:preprintSymmetryReductionOfStatesI}.)}
  For all $\mu \in {]0,\infty[}$, the quadratic module $\R_{0,\mu}$ of the $\momentmap$-reduction of $\Polynomials_0(\CC^{1+n})$ at $\mu$ is
  \begin{equation}
    \R_{0,\mu}
    =
    \big(\Polynomials_0(\CC^{1+n})^{\lie u_1}\big)^+_\Hermitian + \big( \genSId{\momentmap-\mu} \big)_\Hermitian
    =
    \set[\big]{f \in \big(\Polynomials(\CC^{1+n})^{\lie u_1}\big)_\Hermitian }{ f(w) \ge 0 \textup{ for all }w \in \Levelset_\mu }
  \end{equation}
  with $\Levelset_\mu \coloneqq \set{w\in \CC^{1+n}}{\momentmap(w) = \mu}$ the $\mu$-levelset of $\momentmap$.
  Moreover, $\supp_\CC \R_{0,\mu} = \genSId{\momentmap-\mu}$ and
  $\Polynomials_0(\CC^{1+n})_{\mu\mred} = \Polynomials_0(\CC^{1+n})^{\lie u_1} / \genSId{\momentmap-\mu}$
  is isomorphic to the ordered $^*$\=/algebra of polynomial functions on $\CC\PP^n = \Levelset_\mu / \group{U}(1)$ with the pointwise order.
  Finally, $\R_{0,\mu}$ can also be described in a mostly algebraic way as
  \begin{equation}
    \R_{0,\mu}
    =
    \set[\big]{
      f \in (\Polynomials(\CC^{1+n})^{\lie u_1})_\Hermitian 
    }{
      f + \epsilon \Unit \in \big(\Polynomials_0(\CC^{1+n})^{\lie u_1} \big)^{++}_\Hermitian + \big( \genSId{\momentmap-\mu} \big)_\Hermitian\textup{ for all }\epsilon \in {]0,\infty[}
    }
    \punkt
    \label{eq:schmuedgen}
  \end{equation}
\end{proposition}
Identity \eqref{eq:schmuedgen} is essentially Schmüdgen's Positivstellensatz applied to the special case of the compact real algebraic set $\CC\PP^n$.
Spelled out in detail, \eqref{eq:schmuedgen} says that every Hermitian element $f$ of $\Polynomials_0(\CC^{1+n})^{\lie u_1}$ that fulfils $\dupr{\omega}{f} > 0$ for all
$\mu$-eigenstates $\omega$ of $\momentmap$, or equivalently, $f(w)> 0$ for all points $w$ of the compact $\mu$-levelset of $\momentmap$,
can be expressed as
\begin{equation}
  f = \sum_{j=1}^\ell \cc{g_j} g_j + (\momentmap-\mu\Unit) h
\end{equation}
with suitable $\ell \in \NN_0$, $g_1, \dots, g_\ell \in \Polynomials_0(\CC^{1+n})^{\lie u_1}$ and $h \in (\Polynomials_0(\CC^{1+n})^{\lie u_1})_\Hermitian$.
It is noteworthy that this result is in some sense optimal: By \cite[Proposition~6.1]{scheiderer:sumsOfSquaresOfRegularFunctionsOnRealAlgebraicVarieties} there does not exist a denominator-free non-strict Positivstellensatz for $\CC\PP^n$ if $n \geq 2$, i.e.\ there is an element of $\R_{0,\mu}$ which cannot be expressed as a sum of Hermitian squares plus an element of the ideal generated by $\momentmap-\mu$ (with respect to the pointwise product).
An explicit example is:

\begin{proposition} \label{prop:complexMotzkin}
	For $n \geq 2$ and $\mu \in {]0,\infty[}$, the element
	\begin{equation}
		f \coloneqq -\frac 1 {16}\big( z_0^2 \cc z_1^2 - \cc z_0^2 z_1^2 \big)^2 
		\big(\abs {z_0}^2 \abs {z_1}^2 - 3 \abs {z_1}^4\big) + \abs {z_1}^{12} \in 
		\mathcal R_{0,\mu}
	\end{equation}
is not in $\big(\Polynomials_0(\CC^{1+n})^{\lie u_1} \big)^{++}_\Hermitian + \big( \genSId{\momentmap-\mu} \big)_\Hermitian$.
\end{proposition}

\begin{proof}
	Note that $f$ is a homogenized version of the Motzkin polynomial. Indeed,
	$f(x+\I y, 1, \dots, 1) = x^2 y^2(x^2+y^2-3)+1$ and if $w_1 \neq 0$, then $f(w_0, \dots, w_n) = \abs{w_1}^{12} f(\frac{w_0}{w_1}, 1, \frac{w_2}{w_1}, \dots, \frac{w_n}{w_1})$.
	Since the Motzkin polynomial is pointwise positive and not a sum of squares, this implies that
	$f \in \mathcal R_{0,\mu}$ but $f \notin \big(\Polynomials_0(\CC^{1+n})^{\lie u_1} \big)^{++}_\Hermitian$.
	
	Assume that $f - \sum_{i=1}^k g_i^2 \in \genSId{\momentmap - \mu}$ with 
	$g_i = \sum_{K,L \in \mathbb N_0^{1+n}, \abs{K} = \abs L} g_{i,K,L} z^K \cc z^L \in \Polynomials(\mathbb C^{1+n})^{\lie u_1}$.
	Let $d \in \mathbb N$, $d \geq 3$ be such that the total degree of each $g_i$, $1 \leq i \leq k$, is less than or equal to $2d$.
	Write $\hat g_i = \sum_{K,L \in \mathbb N_0^{1+n}, \abs K = \abs L} g_{i,K,L} z^K \cc z^L (\momentmap / \mu)^{d-\abs K} 
	\in \Polynomials(\mathbb C^{1+n})^{d,d}$ and note that $\hat g_i - g_i \in \genSId{\momentmap - \mu}$ 
	since $(\momentmap/\mu)^\ell - \Unit = (\momentmap/\mu - \Unit) \sum_{j=0}^{\ell-1} (\momentmap/\mu)^j \in \genSId{\momentmap-\mu}$
	holds for all $\ell \in \mathbb N_0$.
	Then $(\momentmap/\mu)^{2d - 6} f - \sum_{i=1}^k \hat g_{i}^2 \in \genSId{\momentmap-\mu}$.
	But $(\momentmap/\mu)^{2d - 6} f- \sum_{i=1}^k \hat g_{i}^2$ is homogeneous, and the only homogeneous element of $\genSId{\momentmap - \mu}$ is $0$,
	so $(\momentmap/\mu)^{2d - 6} f$ and also $\momentmap^{2d-6} f$ must be sums of squares.
	Setting $z_2 = 1$ and $z_3 = \dots = z_n = 0$, we obtain that $(1+\abs{z_0}^2+\abs{z_1}^2)^{2d - 6} f$ must also be a sum of squares, 
	and so must be the lowest order $f$,
	contradicting the first paragraph.
\end{proof}
Note that the situation is different if $n=1$, in which case it follows easily from \cite[Corollary 3.12]{scheiderer:sumsOfSquaresOnRealAlgebraicSurfaces} that every element of $\mathcal R_{0,\mu}$ can be written as a sum of Hermitian squares plus an element of the ideal generated by $\momentmap-\mu$.

One missing piece in the analogy between the cases $\hbar = 0$ and $\hbar > 0$ is to give an algebraic
description of the quadratic module $\R_{\hbar, \mu}$.
In contrast to Proposition~\ref{prop:complexMotzkin} we will even show:
\begin{theorem} \label{theorem:main}
  For all $\mu \in {[0,\infty[}$ and all $\hbar \in {]0,\infty[}$, the identity
  \begin{equation}
    \R_{\hbar,\mu}
    =
    \big( \Polynomials_\hbar(\CC^{1+n})^{\lie u_1}\big) ^{++}_\Hermitian + \big( \genSId{\momentmap-\mu} \big)_\Hermitian
    \label{eq:main}
  \end{equation}
  holds.
\end{theorem} 
This means that every Hermitian element $f$ of $\Polynomials_\hbar(\CC^{1+n})^{\lie u_1}$ that fulfils $\dupr{\omega}{f} \ge 0$ for all
$\mu$-eigenstates $\omega$ of $\momentmap$, or equivalently, $\skalhbar{\psi}{\pi_\hbar(f)(\psi)} \ge 0$ for all $\mu$-eigenvectors $\psi$ of $\pi_\hbar(\momentmap)$,
can be expressed as
\begin{equation}
  f = \sum_{j=1}^\ell \cc{g_j} \star_\hbar g_j + (\momentmap-\mu\Unit) \star_\hbar h
\end{equation}
with suitable $\ell \in \NN_0$, $g_1, \dots, g_\ell \in \Polynomials_\hbar(\CC^{1+n})^{\lie u_1}$ and $h \in (\Polynomials_\hbar(\CC^{1+n})^{\lie u_1})_\Hermitian$.
In contrast to the ``strict'' Positivstellensatz \eqref{eq:schmuedgen} that one obtains for $\hbar = 0$,
the non-commutative Positivstellensatz of Theorem~\ref{theorem:main} is even a ``non-strict'' one, 
i.e.~a Nichtnegativstellensatz. Theorem~\ref{theorem:main} thus is stronger than the description of $\R_{\hbar,\mu}$
that one could obtain by applying the techniques of \cite{schmuedgen:StrictPositivstellensatzForEnvelopingAlgebras} for $\lie{su}_{1+n}$.

While the inclusion ``$\supseteq$'' in \eqref{eq:main} clearly is fulfilled,
see e.g.~Corollary~\ref{corollary:neu}, the converse inclusion ``$\subseteq$'' will be proven in the next Section~\ref{sec:proof}.
Before doing so, it might be worthwhile pointing out that Theorem~\ref{theorem:main} is not just a trivial consequence of the simple fact that the reduced algebra
$\Polynomials_\hbar(\CC^{1+n})_{\mu\mred}$ is a finite dimensional $C^*$\=/algebra in which all positive Hermitian elements have a square root. This observation
only yields:
\begin{proposition} \label{proposition:trivial}
  For all $\hbar \in {]0,\infty[}$ and all $\mu \in {[0,\infty[}$, the identity
  \begin{equation}
    \R_{\hbar,\mu}
    =
    \big( \Polynomials_\hbar(\CC^{1+n})^{\lie u_1}\big) ^{++}_\Hermitian + \supp \R_{\hbar,\mu}
  \end{equation}
  holds.
\end{proposition}
\begin{proof}
  Recall that $[\argument]_\mu \colon \Polynomials_\hbar(\CC^{1+n})^{\lie u_1} \to \Polynomials_\hbar(\CC^{1+n})_{\mu\mred} = \Polynomials_\hbar(\CC^{1+n})^{\lie u_1} / \supp_\CC \R_{\hbar,\mu}$
  is the canonical projection onto the quotient. Given $f \in \R_{\hbar,\mu}$, then $[f]_\mu$ is a positive Hermitian element of
  the finite dimensional $C^*$\=/algebra $\Polynomials_\hbar(\CC^{1+n})_{\mu\mred}$, and therefore there exists $[g]_\mu \in \Polynomials_\hbar(\CC^{1+n})_{\mu\mred}$
  with representative $g\in \Polynomials_\hbar(\CC^{1+n})^{\lie u_1}$ such that $[f]_\mu = [g^* \star_\hbar g]_\mu$. As a consequence,
  $f= g^* \star_\hbar g + h$ with $h \coloneqq f- g^*\star_\hbar g \in \supp \R_{\hbar,\mu}$.
\end{proof}
Note that $\supp_\CC \R_{\hbar,\mu} \neq \genSId{\momentmap-\mu}$ for $\hbar \neq 0$, see also the 
discussion in Section~\ref{sec:application}.
The relation between the $^*$\=/ideals $\supp_\CC \R_{\hbar,\mu}$ and $\genSId{\momentmap-\mu}$ can 
be made more explicit.
This requires the following Lemma and the definition of the falling factorial
\begin{equation}
  \falling{x}{k} \coloneqq \prod_{j=0}^{k-1}(x-j)
  \label{eq:fallingFactorialDef}
\end{equation}
for $x\in \RR$ and $k\in \NN_0$.
\begin{lemma} \label{lemma:moduloSId1}
	We have $\momentmap^k g - \hbar^k \falling{\frac \mu \hbar - \ell}{k} g  \in \genSId{\momentmap-\mu}$
	for all $\hbar \in {]0,\infty[}$, $\mu \in \RR$, $k, \ell \in \NN_0$ and $g\in \Polynomials^{\ell,\ell}_\hbar(\CC^{1+n})$.
	Note that here, like always, juxtaposition and exponentiation $\argument^k$ as in $\momentmap^k g$ refer to pointwise multiplication,
	but $\genSId{\momentmap-\mu}$ denotes the generated $^*$\=/ideal with respect to the product $\star_\hbar$.
\end{lemma}
\begin{proof}
	This can be easily proven by induction over $k$, using that 
	\begin{equation*}
		(\momentmap - \mu) \star_\hbar (\momentmap^{k-1} g) 
		=
		\momentmap^k g + \hbar (k-1+\ell) \momentmap^{k-1} g - \mu \momentmap^{k-1} g
		=
		\momentmap^k g - \hbar \Big(\frac \mu \hbar - \ell - (k - 1)\Big) \momentmap^{k-1} g \punkt
	\end{equation*}
\end{proof}
\begin{proposition} \label{proposition:kerR}
  For all $\hbar \in {]0,\infty[}$ and all $k \in \NN_0$, the $^*$\=/ideal $\supp_\CC \R_{\hbar,\hbar k}$ of $\Polynomials_\hbar(\CC^{1+n})^{\lie u_1}$
  is generated by the union $\{\momentmap-\hbar k\} \cup \Polynomials^{k+1,k+1}(\CC^{1+n})$.
\end{proposition}
\begin{proof}
  From Equation~\eqref{eq:Rmu} in Corollary~\ref{corollary:neu} it follows immediately that $\pm (\momentmap-\hbar k \Unit) \in \R_{\hbar,\hbar k}$,
  and also that $\pm f \in \R_{\hbar,\hbar k}$ for every $f\in \Polynomials^{k+1,k+1}(\CC^{1+n})$ because $\pi_\hbar(f)(g) = 0$
  for all $g\in \Polynomials^{k,0}(\CC^{1+n})$.
  
  Conversely, let $f = \sum_{\ell = 0}^\infty f_\ell \in \supp_\CC \R_{\hbar,\hbar k}$ with homogeneous components $f_\ell \in \Polynomials^{\ell,\ell}(\CC^{1+n})$
  be given. We show that $\sum_{\ell=k+1}^\infty f_\ell$ is in the $^*$\=/ideal generated by $\Polynomials^{k+1,k+1}(\CC^{1+n})$
  and that $\sum_{\ell = 0}^k f_\ell$ lies in the $^*$\=/ideal generated by $\genSId{\momentmap-\hbar k}$:
  
  For $L,L' \in \NN_0^{1+n}$ with $\ell \coloneqq \abs{L} = \abs{L'} > k$ there are $M,M',N,N' \in \NN_0^{1+n}$ with $\abs{M} = \abs{M'} = k+1$ such that 
  $L = M+N$ and $L'=M'+N'$. Using Proposition~\ref{proposition:starhbarbasics} one easily checks that this way,
  $z^{L} \cc{z}^{L'} - z^M \cc{z}^{M'} \star_\hbar z^N \cc{z}^{N'} \in \bigoplus_{r=\ell - \min\{k+1,\ell-(k+1)\}}^{\ell-1} \Polynomials^{r,r}(\CC^{1+n})$.
  Note that $\ell - \min\{k+1,\ell-(k+1)\} \ge k+1$.
  Starting with the highest non-vanishing component one can thus show that $\sum_{\ell=k+1}^\infty f_\ell$ is an element of the
  $^*$\=/ideal of $\Polynomials_\hbar(\CC^{1+n})^{\lie u_1}$ that is generated by $\Polynomials^{k+1,k+1}(\CC^{1+n})$. Now define
  \begin{align*}
    g &\coloneqq \sum_{\ell = 0}^k \big( \hbar^{k-\ell} (k-\ell)! \big)^{-1} \momentmap^{k-\ell} f_\ell \in \Polynomials^{k,k}(\CC^{1+n})
    \komma
  \end{align*}
  then $h \coloneqq g-\sum_{\ell = 0}^k f_\ell \in \genSId{\momentmap-\hbar k}$ by Lemma~\ref{lemma:moduloSId1}.
  Furthermore, as $g = f+ h - \sum_{\ell = k+1}^\infty f_\ell$ with $h - \sum_{\ell = k+1}^\infty f_\ell = g-f \in \supp_\CC \R_{\hbar,\hbar k}$,
  $g$ is the unique element of $\Polynomials^{k,k}(\CC^{1+n})$ that coincides with $f$ modulo $\supp_\CC \R_{\hbar,\hbar k}$,
  see Corollary~\ref{corollary:neu}. But this means that $g=0$ because $f \in \supp_\CC \R_{\hbar,\hbar k}$, so $\sum_{\ell = 0}^k f_\ell = - h \in \genSId{\momentmap-\hbar k}$.
\end{proof}
For $\mu \notin \set{\hbar k}{k\in \NN_0}$, of course, $\supp_\CC \R_{\hbar,\mu} = \Polynomials_{\hbar}(\CC^{1+n})^{\lie u_1}$
is the $^*$\=/ideal generated by $\Unit$, or equivalently, by $\{\momentmap-\hbar k\} \cup \Polynomials^{0,0}(\CC^{1+n})$.
\section{Proof of the Main Theorem} \label{sec:proof}
In order to construct representations of positive Hermitian elements as sums of Hermitian squares, 
certain invariant functionals that one obtains by averaging over the $\group{U}(1+n)$-action will be helpful:

\begin{definition}
  For all $k\in \NN_0$, the linear functional $\invstate[k] \colon \Polynomials^{k,k}(\CC^{1+n}) \to \CC$ is defined as the one that fulfils
  \begin{equation}
    \dupr{\invstate[k]}{z^K \cc{z}^L} \coloneqq \delta_{K,L} \frac{K! n!}{(k+n)!}
  \end{equation}
  for all $K,L\in \NN^{1+n}_0$ with $\abs{K} = \abs{L} = k$.
\end{definition}
The crucial properties of these functionals are:
\begin{proposition}
  For all $k\in \NN_0$, the linear functional $\invstate[k]$ on $\Polynomials^{k,k}(\CC^{1+n})$ fulfils
  \begin{align}
    \dupr{\invstate[k]}{\momentmap^k} &= 1
  \shortintertext{and it is $\group{U}(1+n)$-invariant, i.e.}
    \dupr{\invstate[k]}{f\racts u} &= \dupr{\invstate[k]}{f}
  \end{align}
  holds for all $f\in \Polynomials^{k,k}(\CC^{1+n})$ and all $u \in \group{U}(1+n)$.
\end{proposition}
\begin{proof}
  The multinomial formula for $\momentmap^k$ together with the definition of $\invstate[k]$ yield
  \begin{equation*}
    \dupr{\invstate[k]}{\momentmap^k}
    =
    \sum_{K \in \NN_0^{1+n}, \abs{K}=k} \frac{k!}{K!} \dupr{\invstate[k]}{ z^K \cc{z}^K }
    =
    \sum_{K \in \NN_0^{1+n}, \abs{K}=k} \frac{k! n!}{(k+n)!}
    =
    1
    \punkt
  \end{equation*}
  In order to check the $\group{U}(1+n)$-invariance, recall that $\group{U}(1+n)$ acts on $\Polynomials^{k,0}(\CC^{1+n})$ by pullbacks
  and that this action is unitary with respect to the inner product of $\PolyHol(\CC^{1+n})$, for all $\hbar \in {]0,\infty[}$. 
  We can consider $\hbar = 1$ in the following, then $\set{(K!)^{-1/2} z^K}{K\in \NN_0^{1+n} \textup{ with }\abs{K}=k}$
  is a $\skalhbarIsOne{\argument}{\argument}$-orthonormal basis of $\Polynomials^{k,0}(\CC^{1+n})$.
  So let $u\in \group{U}(1+n)$ be given, then there exists a unitary matrix $( \rho(u)_{K,L} )_{K,L}$ representing
  the action of $u$ on $\Polynomials^{k,0}(\CC^{1+n})$ in this basis, i.e.
  \begin{equation*}
    (K!)^{-1/2} z^K \racts u = \sum_{M\in \NN_0^{1+n}, \abs{M} = k} \rho(u)_{K,M} (M!)^{-1/2} z^M
  \end{equation*}
  holds for all $K \in \NN_0^{1+n}$ with $\abs{K} = k$. Using that $\group{U}(1+n)$ also acts on $\Polynomials^{k,k}(\CC^{1+n})$
  by pullbacks, one therefore finds that
  \begin{align*}
    \dupr[\big]{\invstate[k]}{ (z^K \cc{z}^L) \racts u }
    &=
    \dupr[\big]{\invstate[k]}{( z^K \racts u ) ( \cc{z^L \racts u} ) }
    \\
    &=
    (K!)^{1/2} (L!)^{1/2} \sum_{\substack{M,N\in \NN_0^{1+n} \\ \abs{M}=\abs{N}=k}} \dupr[\big]{\invstate[k]}{ \big( \rho(u)_{K,M} (M!)^{-1/2} z^M\big) \big( \cc{ \rho(u)_{L,N} (N!)^{-1/2} z^N}\big) }
    \\
    &=
    (K!)^{1/2} (L!)^{1/2} \sum_{M\in \NN_0^{1+n}, \abs{M}=k} \rho(u)_{K,M} \,\cc{ \rho(u)_{L,M}} \frac{ n!}{(k+n)!}
    \\
    &=
    (K!)^{1/2} (L!)^{1/2} \delta_{K,L} \frac{ n!}{(k+n)!}
    \\
    &=
    \dupr{\invstate[k]}{ z^K \cc{z}^L}
  \end{align*}
  holds for all $K,L\in \NN_0^{1+n}$ with $\abs{K}=\abs{L} = k$.
\end{proof}
\begin{corollary} \label{corollary:invstateK}
  For all $k\in \NN_0$ and all $f\in \Polynomials^{k,k}(\CC^{1+n})$, the identity
  \begin{equation}
     \momentmap^k \dupr{\invstate[k]}{f} = \integral{u\in \group{U}(1+n)}{ (f\racts u) }{ \nu(u)}
  \end{equation}
  holds, where $\nu$ denotes the unique right-invariant volume form on $\group{U}(1+n)$ that fulfils $\integral{\group{U}(1+n)}{}{\nu} = 1$.
\end{corollary}
\begin{proof}
  Given $f\in \Polynomials^{k,k}(\CC^{1+n})$, then let $f_\av \coloneqq \integral{u\in \group{U}(1+n)}{ (f\racts u) }{ \nu(u)} \in \Polynomials^{k,k}(\CC^{1+n})$.
  Using the right-invariance of $\nu$ one finds that
  $f_\av \racts u' = \integral{u\in \group{U}(1+n)}{ (f\racts uu') }{ \nu(u)} = f_\av$ holds for all $u'\in \group{U}(1+n)$,
  i.e. $f_\av$ is $\group{U}(1+n)$-invariant.
  As all such $\group{U}(1+n)$-invariant elements of $\Polynomials^{k,k}(\CC^{1+n})$ necessarily are scalar multiples of $\momentmap^k$,
  there exists $\alpha \in \CC$ such that $f_\av = \alpha \momentmap^k$.
  It now follows that
  \begin{equation*}
    \momentmap^k\dupr{\invstate[k]}{f}
    =
    \momentmap^k \integral{u\in\group{U}(1+n)}{ \dupr{\invstate[k]}{f\racts u} }{\nu(u)}
    =
    \momentmap^k \dupr[\big]{\invstate[k]}{ f_\av }
    =
    \alpha  \momentmap^k \dupr[\big]{\invstate[k]}{ \momentmap^k }
    =
    \alpha \momentmap^k 
    =
    f_\av
    \punkt
  \end{equation*}
\end{proof}
It thus makes sense to define:
\begin{definition}
  We define the \emph{averaging operator} $\argument_\av \colon \bigoplus_{k\in \NN_0} \Polynomials^{k,k}(\CC^{1+n}) \to \bigoplus_{k\in \NN_0} \Polynomials^{k,k}(\CC^{1+n})$,
  \begin{equation}
    \sum\nolimits_{k\in \NN_0} f_k \mapsto \Big( \sum\nolimits_{k\in \NN_0} f_k \Big)_\av \coloneqq \sum\nolimits_{k\in \NN_0} \momentmap^k \dupr{\invstate[k]}{f_k}
    \punkt
  \end{equation}
\end{definition}
\begin{proposition} \label{proposition:avfinitecomb}
  For all $f\in \bigoplus_{k\in \NN_0} \Polynomials^{k,k}(\CC^{1+n})$ there exist $d\in \NN_0$,
  $u_1,\dots, u_d \in \group{U}(1+n)$ and $\lambda_1,\dots, \lambda_d \in {[0,1]}$ with $\sum_{j=1}^d \lambda_j = 1$
  such that
  \begin{equation}
    f_\av = \sum_{j=1}^d \lambda_j f\racts u_j
    \punkt
  \end{equation}
\end{proposition}
\begin{proof}
  Given $f = \sum_{k=0}^\infty f_k \in \bigoplus_{k\in \NN_0} \Polynomials^{k,k}(\CC^{1+n})$ with homogeneous components $f_k \in \Polynomials^{k,k}(\CC^{1+n})$,
  then it follows from Corollary~\ref{corollary:invstateK} that $f_\av = \integral{u\in\group{U}(1+n)}{(f\racts u)}{\nu(u)}$, where $\nu$ is again the unique normalized right-invariant
  volume form on $\group{U}(1+n)$, and where the integral is taken in the finite dimensional vector space $V \coloneqq \bigoplus_{k=0}^{k_{\max}} \Polynomials^{k,k}(\CC^{1+n})$
  with $k_{\max} \in \NN_0$ sufficiently large such that $f \in V$. From $\nu$ being normalized it follows that $f_\av$
	lies in the closure of the convex hull of the compact subset $S \coloneqq \set{ f \racts u }{u \in \group U(1+n)}$ of $V$.
	By Caratheodory's theorem, every element of the convex hull of $S$ can be expressed as a convex combination of $d \coloneqq 1 + \dim V$ elements of $S$.
	Thus the convex hull of $S$ is the image of the continuous map $\Delta^{(d-1)} \times S^{d} \to V$,
	$\big(\lambda_1, \dots, \lambda_{d}, f \racts u_1, \dots, f \racts u_d\big) \mapsto \sum_{j=1}^{d} \lambda_j f \racts u_j$
	with $\Delta^{(d-1)} \coloneqq \set[\big]{(\lambda_1, \dots, \lambda_{d}) \in [0,1]^d}{ \sum_{j=1}^d \lambda_j = 1}$ the compact $(d-1)$-simplex.
	However, this convex hull, being the image of the compact space $\Delta^{(d-1)} \times S^{d}$ under a continuous map,
	is compact and therefore is already closed, and consequently contains $f_\av$. So $f_\av = \sum_{j=1}^{d} \lambda_j f \racts u_j$
	for suitable $(\lambda_1,\dots, \lambda_d) \in \Delta^{(d-1)}$ and $u_1,\dots,u_d \in \group{U}(1+n)$.
\end{proof}
One can also combine the functionals $\invstate[k]$ into one especially useful functional on $\Polynomials_\hbar(\CC^{1+n})^{\lie u_1}$:
\begin{definition} \label{definition:invstate1}
  For all $\hbar \in {]0,\infty[}$ and all $\mu \in \RR$, define $\invstateSum \colon \Polynomials_\hbar(\CC^{1+n})^{\lie u_1} \to \CC$,
  \begin{equation}
    f \mapsto \dupr{\invstateSum}{f} \coloneqq \sum_{k=0}^\infty \hbar^k \Big( \frac{\mu}{\hbar} \Big)_{\downarrow, k} \dupr{\invstate[k]}{ f_k }
    ,
  \end{equation}
  where $f_k \in \Polynomials^{k,k}(\CC^{1+n})$ are the homogeneous components of $f$ and with the falling factorial $\falling{\argument}{k}$
  from \eqref{eq:fallingFactorialDef}.
\end{definition}
\begin{proposition} \label{proposition:invstateProp}
  Let $\hbar \in {]0,\infty[}$, $\mu \in \RR$ and $f\in \Polynomials_\hbar(\CC^{1+n})^{\lie u_1}$ be given, then
  \begin{equation}
    \dupr{\invstateSum}{f} \Unit - f_\av \in \genSId{\momentmap-\mu} .
  \end{equation}
\end{proposition}
\begin{proof}
	From $f_\av = \sum_{k = 0}^\infty (f_\av)_k = \sum_{k=0}^\infty \momentmap^k \dupr{\invstate[k]}{f_k}$
	with $(f_\av)_k \in \Polynomials^{k,k}(\CC^{1+n})$ the homogeneous components of $f_\av$,
	together with Lemma~\ref{lemma:moduloSId1} for $\ell \coloneqq 0$, $g\coloneqq \Unit$, it follows that
	\begin{equation*}
	  \dupr{\invstateSum}{f} \Unit - f_\av
	  =
	  \sum_{k=0}^\infty \bigg( \hbar^k \Big( \frac{ \mu }{ \hbar }\Big)_{\downarrow, k} \Unit - \momentmap^k\bigg) \dupr{\invstate[k]}{ f_k }
	  \in
	  \genSId{\momentmap-\mu}
	  \punkt
	\end{equation*}
\end{proof}
This property of the functional $\invstateSum$ can be used in order to determine a representation of $-\Unit$ as a sum of Hermitian squares,
up to a contribution of $\genSId{\momentmap-\mu}$,
provided that one finds $f\in \Polynomials_\hbar(\CC^{1+n})^{\lie u_1}$ for which $\dupr{\invstateSum}{f^* \star_\hbar f} < 0$ holds.
In order to explicitly calculate this expression, we need:
\begin{lemma} \label{lemma:fallingFactorials}
	For all $x, y \in \CC$ and all $k \in \NN_0$ we have
	\begin{equation} \label{eq:fallingFactorialFormula1}
		\falling{x+y}{k} = \sum_{t=0}^k \binom k t \falling x t \falling y {k-t} \punkt
	\end{equation}
	In particular,
	\begin{equation} \label{eq:summingFallingFactorials1}
		\sum_{t=0}^k \binom{k}{t} \frac{\falling{z}{2k-t}}{(2k -t +n)!} 
		=
		\frac{ \falling z k \falling{z+k+n}{k}  }{(2 k+n)!} 
		\punkt
	\end{equation}
	holds for all $z \in \CC$ and $k \in \NN_0$.
\end{lemma}

\begin{proof}
	Using the identity $\binom{k+1}{t} =\binom{k}{t-1} + \binom{k}{t}$ one easily proves \eqref{eq:fallingFactorialFormula1} by induction over $k$.
	Then
	\begin{equation*}
		\sum_{t=0}^k \binom{k}{t} \frac{\falling{z}{2k-t} }{(2k -t +n)!} 
    =
		\frac {\falling z k} {(2k+n)!} \sum_{t=0}^k \binom{k}{t} \falling{z-k}{k-t} \falling{2k+n}{t}
		=
		\frac {\falling z k} {(2k+n)!} \falling{z+k+n}{k} \punkt
	\end{equation*}
\end{proof}
\begin{proposition} \label{proposition:minusOne}
  For all $\hbar \in {]0,\infty[}$ and all $\mu \in \RR \setminus \big( \set{\hbar k}{k \in \NN_0} \cup \set{-\hbar (1+n + k) }{k \in \NN_0 } \big)$
  one has
  \begin{equation}
    {- \Unit} \in \big( \Polynomials_\hbar(\CC^{1+n})^{\lie u_1} \big)^{++}_\Hermitian + \big(\genSId{\momentmap-\mu}\big)_\Hermitian
    \punkt
  \end{equation}
\end{proposition}

\begin{proof}
  For all $k\in \NN_0$, $\mu \in \RR$ we calculate, using identity \eqref{eq:summingFallingFactorials1} from the previous Lemma~\ref{lemma:fallingFactorials}:
  \begin{align*}
    \dupr[\big]{\invstateSum}{\big(z_0^k \cc z_1^k\big)^* \star_\hbar \big(z_0^k \cc z_1^k\big)}
    &=
    \dupr[\Big]{\invstateSum}{ \sum\nolimits_{t=0}^k \frac{\hbar^t(k!)^2}{t!((k-t)!)^2} z_0^{k-t} \cc z_0^{k-t}z_1^k\cc z_1^k }
    \\
    &=
    \sum_{t=0}^k \frac{\hbar^{2k}(k!)^2}{t!((k-t)!)^2} \bigg(\frac{\mu}{\hbar}\bigg)_{\downarrow,2k-t} \dupr[\big]{\invstate[2k-t]}{ z_0^{k-t} \cc z_0^{k-t}z_1^k\cc z_1^k }
    \\
    &=
    \sum_{t=0}^k \frac{\hbar^{2k}(k!)^3 n!}{t!(k-t)!(2k-t+n)!} \bigg(\frac{\mu}{\hbar}\bigg)_{\downarrow,2k-t} 
    \\
    &=
    \hbar^{2k}(k!)^2 n! \sum_{t=0}^k \binom{k}{t} \frac{(\mu / \hbar)_{\downarrow, 2k-t}}{ (2k-t+n)! }
    \\
    &=
    \hbar^{2k}(k!)^2 n! \frac{(\mu / \hbar)_{\downarrow, k} (\mu / \hbar + n + k)_{\downarrow, k} }{ (2k+n)! }
    \punkt
  \end{align*}
  If $\mu \notin \set{\hbar k}{k \in \NN_0} \cup \set{-\hbar (1+n + k) }{k \in \NN_0}$,
  then one can find an exponent $k\in \NN$ for which $\alpha \coloneqq -\dupr{\invstateSum}{(z_0^k \cc z_1^k)^* \star_\hbar z_0^k \cc z_1^k} > 0$:
  Indeed, if $\mu > 0$, then one can choose $k$ as the smallest element of $\NN$ for which $\mu / \hbar - (k-1) < 0$,
  and if $\mu < 0$, then choosing $k$ as the smallest element of $\NN$ fulfilling $\mu / \hbar + n+k > 0$ works.
  
  Combining Propositions~\ref{proposition:avfinitecomb} and \ref{proposition:invstateProp} now shows that
  there exist $h \in \genSId{\momentmap-\mu}$, $d\in \NN_0$,
  $u_1,\dots, u_d \in \group{U}(1+n)$ and $\lambda_1,\dots, \lambda_d \in {[0,1]}$ with $\sum_{j=1}^d \lambda_j = 1$
  such that
  \begin{equation*}
    -\alpha \Unit - h
    =
    \sum_{j=1}^d \lambda_j \big((z_0^k \cc z_1^k)^* \star_\hbar z_0^k \cc z_1^k\big)\racts u_j 
    =
    \sum_{j=1}^d \lambda_j (z_0^k \cc z_1^k\racts u_j)^* \star_\hbar (z_0^k \cc z_1^k \racts u_j)
    \in
    \big(\Polynomials_\hbar(\CC^{1+n})^{\lie u_1}\big)^{++}_\Hermitian
  \end{equation*}
  holds. Rescaling by $\alpha^{-1}$ one finds that $- \Unit \in ( \Polynomials_\hbar(\CC^{1+n})^{\lie u_1} )^{++}_\Hermitian + (\genSId{\momentmap-\mu})_\Hermitian$.
\end{proof}
This proves Theorem~\ref{theorem:main} for momenta $\mu \in {[0,\infty[} \setminus \set{\hbar k}{k \in \NN_0}$. For
$\mu \in \set{\hbar k}{k \in \NN_0}$, however, we need a different approach:
\begin{lemma} \label{lemma:posSquares}
  Let $\hbar \in {]0,\infty[}$ and let $f \in \Polynomials^{k,0}(\CC^{1+n})$ be a homogeneous holomorphic polynomial of degree $k \in \NN$.
  Then $f \cc f \in \big(\Polynomials_\hbar(\CC^{1+n})^{\lie u_1} \big)_\Hermitian^{++} + \big( \genSId{\momentmap-\hbar (k-1)} \big)_\Hermitian$.
\end{lemma}
\begin{proof}
  From 
  \begin{equation*}
    \cc z_0^{k} \star_\hbar z_0^{k}
    =
    \sum_{t=0}^k \frac{\hbar^t ( k! )^2 }{t! ((k-t)!)^2 } (z_0\cc{z_0})^{k-t}
    \quad\quad\text{and}\quad\quad
    \dupr[\big]{\invstate[k-t]}{ (z_0\cc{z_0})^{k-t} } = \frac{(k-t)! n!}{ (k-t+n)! }
  \end{equation*}
  it follows that
  \begin{equation*}
    \big( \cc z_0^{k} \star_\hbar z_0^{k} \big)_{\av}
    =
    \sum_{t=0}^k \frac{\hbar^t (k!)^2 n! }{ t! (k-t)! (k-t+n)! } \momentmap^{k-t}
    \punkt
  \end{equation*}
  By Proposition~\ref{proposition:avfinitecomb} there exist 
  $d\in \NN_0$, $u_1,\dots,u_d \in \group{U}(1+n)$ and $\lambda_1,\dots, \lambda_d \in {[0,1]}$
  with $\sum_{j=1}^d \lambda_j = 1$ such that
  \begin{equation*}
    \sum_{j=1}^d \lambda_j (\cc z_0^{k} \racts u_j) \star_\hbar (z_0^{k} \racts u_j)
    =
    \sum_{j=1}^d \lambda_j (\cc z_0^{k} \star_\hbar z_0^{k}) \racts u_j
    =
    \big( \cc z_0^{k} \star_\hbar z_0^{k} \big)_{\av}
    =
    \sum_{t=0}^k \frac{\hbar^t (k!)^2 n! }{ t! (k-t)! (k-t+n)! } \momentmap^{k-t}
  \end{equation*}
  holds, and consequently
  \begin{align*}
    \sum_{j=1}^d \lambda_j \big(( z_0^{k} \racts u_j) \cc{f}\big)^* \star_\hbar \big(( z_0^{k} \racts u_j) \cc{f}\big)
    &=
    f \star_\hbar \bigg( \sum_{j=1}^d \lambda_j (\cc z_0^{k} \racts u_j) \star_\hbar (z_0^{k} \racts u_j) \bigg) \star_\hbar \cc{f}
    \\
    &=
    f \star_\hbar \bigg( \sum_{t=0}^k \frac{\hbar^t (k!)^2 n! }{ t! (k-t)! (k-t+n)! } \momentmap^{k-t} \bigg) \star_\hbar \cc{f}
    \\
    &=
    \sum_{t=0}^k \frac{\hbar^t (k!)^2 n! }{ t! (k-t)! (k-t+n)! } \momentmap^{k-t} f\cc{f}
    \punkt
  \end{align*}
  As $( z_0^{k} \racts u_j) \cc{f} \in \Polynomials^{k,k}(\CC^{1+n}) \subseteq \Polynomials_\hbar(\CC^{1+n})^{\lie u_1}$,
  the above term is an element of $( \Polynomials_\hbar(\CC^{1+n})^{\lie u_1} )^{++}_\Hermitian$. Define
  \begin{equation*}
    g
    \coloneqq
    \sum_{t=0}^k \frac{\hbar^t (k!)^2 n! }{ t! (k-t)! (k-t+n)! } \big( \momentmap^{k-t} f\cc{f} - \hbar^{k-t} (-1)_{\downarrow,k-t} f\cc{f} \big)
    \komma
  \end{equation*}
  then $g \in ( \genSId{\momentmap-\hbar(k-1)} )_\Hermitian$ by Lemma~\ref{lemma:moduloSId1}. Moreover,
  \begin{align*}
    \sum_{t=0}^k \frac{\hbar^t (k!)^2 n! }{ t! (k-t)! (k-t+n)! } \hbar^{k-t} (-1)_{\downarrow,k-t}
    &=
    \hbar^k (k!)^2 n! \sum_{t=0}^k \frac{ (-1)^{k-t} }{ t! (k-t+n)! } 
    \\
    &=
    \frac{\hbar^k (k!)^2 n!}{(k+n)!} \sum_{t=0}^k (-1)^{k-t} \binom{k+n}{t}
    \\
    &=
    \frac{\hbar^k (k!)^2 n!}{(k+n)!} \binom{k-1+n}{k}
    \\
    &=
    \frac{\hbar^k k! \,n}{k+n}
    \komma
  \end{align*}
  using that $\binom{k+n}{t} = \binom{k-1+n}{t-1} + \binom{k-1+n}{t}$ for all $t\in \{0,\dots,k\}$
  with the convention that $\binom{k-1+n}{-1} \coloneqq 0$. Putting everything together we find that
  \begin{equation*}
    f \cc{f}
    =
    \frac{k+n}{\hbar^k k! \,n} \bigg(\sum_{j=1}^d \lambda_j \big(( z_0^{k} \racts u_j) \cc{f}\big)^* \star_\hbar \big(( z_0^{k} \racts u_j) \cc{f}\big) -g\bigg)
    \in
    \big(\Polynomials_\hbar(\CC^{1+n})^{\lie u_1}\big)^{++}_\Hermitian + \big( \genSId{\momentmap-\hbar(k-1)}\big)_\Hermitian
    \punkt
  \end{equation*}
\end{proof}
\begin{proposition} \label{proposition:kplus1homogeneous}
  For all $k\in \NN$ and all $\hbar \in {]0,\infty[}$ one has 
  \begin{equation}
    \Polynomials^{k,k}(\CC^{1+n}) \subseteq \supp_\CC \Big(\big(\Polynomials_\hbar(\CC^{1+n})^{\lie u_1}\big)^{++}_\Hermitian + \big( \genSId{\momentmap-\hbar(k-1)}\big)_\Hermitian \Big)
    \punkt
  \end{equation}
\end{proposition}
\begin{proof}
  It is sufficient to show that 
  $\I^m z^K \cc{z}^L + \I^{-m} z^L\cc{z}^K \in (\Polynomials_\hbar(\CC^{1+n})^{\lie u_1})^{++}_\Hermitian + ( \genSId{\momentmap-\hbar(k-1)})_\Hermitian$
  for all $K,L\in \NN_0^{1+n}$ with $\abs{K} = \abs{L} = k$ and all $m\in \{0,1,2,3\}$.
  
  First note that $\momentmap^k = (\momentmap - \hbar(k-1)) \star_\hbar \momentmap^{k-1} \in \genSId{\momentmap-\hbar(k-1)}$.
  From this and the previous Lemma~\ref{lemma:posSquares} it follows that
  \begin{equation*}
		- z^M \cc z^M
		=
		\frac{M!}{\abs{M}!}\bigg(
      \sum\nolimits_{\substack{N \in \NN_0^{1+n}\\ \abs N = k, N \neq M}} \frac{\abs N !}{N!} z^N \cc z^N
      -
      \momentmap^{k}
    \bigg)
    \in
    \big(\Polynomials_\hbar(\CC^{1+n})^{\lie u_1} \big)_\Hermitian^{++} + \big( \genSId{\momentmap-\hbar(k-1)} \big)_\Hermitian
	\end{equation*}
	for all $M\in \NN_0^{1+n}$ with $\abs{M} = k$. Making use of Lemma~\ref{lemma:posSquares} again, one finds that
	\begin{equation*}
	  \I^m z^K \cc{z}^L + \I^{-m} z^L\cc{z}^K
	  =
	  (z^K+\I^{-m} z^L)\cc{(z^K+\I^{-m} z^L)}
	  -
	  z^K \cc{z}^K
	  -
	  z^L \cc{z}^L
	\end{equation*}
	is an element of
	$\big(\Polynomials_\hbar(\CC^{1+n})^{\lie u_1} \big)_\Hermitian^{++} + \big( \genSId{\momentmap-\hbar(k-1)} \big)_\Hermitian$
	for all $K,L\in \NN_0^{1+n}$ with $\abs{K} = \abs{L} = k$ and all $m\in \{0,1,2,3\}$.
\end{proof}
This essentially proves Theorem~\ref{theorem:main} for momenta $\mu \in \set{\hbar k}{k \in \NN_0}$. In more detail, we have:

\begin{proof}[of the main Theorem~\ref{theorem:main}]
  Let $\hbar \in {]0,\infty[}$ be given, then we have to show that the inclusion
  $\R_{\hbar,\mu} \subseteq ( \Polynomials_\hbar(\CC^{1+n})^{\lie u_1})^{++}_\Hermitian + (\genSId{\momentmap-\mu})_\Hermitian$ holds:
  
  First assume that $\mu \in {[0,\infty[} \setminus \set{\hbar k}{k\in \NN_0}$. In this case, Proposition~\ref{proposition:minusOne} applies and shows that
  $-\Unit \in ( \Polynomials_\hbar(\CC^{1+n})^{\lie u_1})^{++}_\Hermitian + (\genSId{\momentmap-\mu})_\Hermitian$,
  so $( \Polynomials_\hbar(\CC^{1+n})^{\lie u_1})^{++}_\Hermitian + (\genSId{\momentmap-\mu})_\Hermitian = (\Polynomials_\hbar(\CC^{1+n})^{\lie u_1})_\Hermitian$,
  which certainly contains $\R_{\hbar,\mu}$.
  Now consider the case $\mu = \hbar k$, $k\in \NN_0$. By Proposition~\ref{proposition:kerR},
  the $^*$\=/ideal $\supp_\CC \R_{\hbar,\hbar k}$ is generated by $\{ \momentmap - \hbar k \Unit \} \cup \Polynomials^{k+1,k+1}(\CC^{1+n})$,
  so Proposition~\ref{proposition:kplus1homogeneous} shows that
  \begin{equation*}
    \supp_\CC \R_{\hbar,\hbar k} \subseteq \supp_\CC \!\Big( \big(\Polynomials_\hbar(\CC^{1+n})^{\lie u_1} \big)_\Hermitian^{++} + \big( \genSId{\momentmap-\hbar k} \big)_\Hermitian \Big)
    \punkt
  \end{equation*}
  By Proposition~\ref{proposition:trivial}, this implies $\R_{\hbar,\hbar k} \subseteq (\Polynomials_\hbar(\CC^{1+n})^{\lie u_1} )_\Hermitian^{++} + ( \genSId{\momentmap-\hbar k} )_\Hermitian$.
\end{proof}

\section{Wick Star Product on \texorpdfstring{$\CC\PP^n$}{CPn}} \label{sec:application}

In this section we briefly recall the construction of a deformation quantization of 
$\CC\PP^n$ from \cite{bordemann.brischle.emmrich.waldmann:PhaseSpaceReductionForStarProducts.ExplicitConstruction, bordemann.brischle.emmrich.waldmann:SubalgebrasWithConvergingStarProducts} by reduction of the Wick star product on $\CC^{1+n}$,
keeping $\mu = 1$ fixed and varying $\hbar \in \RR$.
We then determine the $^*$\=/representation theory of the quantized polynomial $^*$\=/algebra on $\CC\PP^n$ that one obtains this way.

Recall that elements of the reduction $\Polynomials_0(\CC^{1+n})_{1\mred}$ can be identified with 
polynomial functions on the real algebraic set $\CC\PP^n$, see e.g.~\cite[Sec.~6]{schmitt.schoetz:preprintSymmetryReductionOfStatesI} for details: 
Any $f \in \Polynomials_0(\CC^{1+n})^{\lie u_1}$ defines a map $\Psi_0(f) \colon \CC\PP^n \to \CC$, $[w] \mapsto \Psi_0(f)([w]) \coloneqq f(w)$,
where $w \in \Levelset_1$ is any representative of $[w] \in \CC\PP^n$ with $\momentmap(w) = 1$. 
The resulting space of polynomials on $\CC\PP^n$ is
$\Polynomials(\CC\PP^n) \coloneqq \set{\Psi_0(f)}{f\in \Polynomials_0(\CC^{1+n})^{\lie u_1} }$. 
The kernel of the map $\Psi_0 \colon \Polynomials_0(\CC^{1+n})^{\lie u_1} \to \Polynomials(\CC\PP^n)$ is just the $^*$\=/ideal $\genSId{\momentmap-1}$
in $\Polynomials_0(\CC^{1+n})^{\lie u_1}$, i.e.~with respect to the pointwise product.
By deforming $\Psi_0$ one can retain this relation for almost all values of $\hbar$ and thus construct a product on $\Polynomials(\CC\PP^n)$ with a rational dependence on $\hbar$:
For $\hbar \in \RR \setminus \{0\}$ define $\Psi_\hbar \colon \Polynomials_\hbar(\CC^{1+n})^{\lie u_1} \to \Polynomials(\CC\PP^n)$,
\begin{equation}
  f \mapsto \Psi_\hbar(f) \coloneqq \sum_{k=0}^\infty \hbar^k \falling[\Big]{\frac 1 \hbar}{k} \Psi_0(f_k)
\end{equation}
with $f_k \in \Polynomials^{k,k}(\CC^{1+n})$ the homogeneous components of $f\in \Polynomials_\hbar(\CC^{1+n})^{\lie u_1}$.
\begin{lemma} \label{lemma:cpnProduct}
	Let $\hbar \in \RR \setminus \big(\{0\} \cup \set[\big]{\frac 1 k}{k \in \NN} \big)$.
	Then the kernel of $\Psi_\hbar$ is precisely the $^*$\=/ideal $\genSId{\momentmap - 1}$ of $\Polynomials_\hbar(\CC^{1+n})^{\lie u_1}$,
	and $\Psi_\hbar$ is a surjective map and fulfils $\Psi_\hbar(f^*)([w]) = \cc{\Psi_\hbar(f)([w])}$ for all $f\in \Polynomials_\hbar(\CC^{1+n})^{\lie u_1}$,  $[w]\in\CC\PP^n$.
\end{lemma}

\begin{proof}
	We have $\genSId{\momentmap - 1} \subseteq \ker \Psi_\hbar$ since 
	\begin{equation*}
	\Psi_\hbar\big((\momentmap - \Unit) \star_\hbar z^K \cc z^L\big) 
	=
	\Psi_\hbar\big(\momentmap z^K \cc z^L - (1 -\hbar \abs K) z^K \cc z^L\big)
	=
	\hbar^{\abs K+1} \falling[\Big]{\frac 1 \hbar}{\abs K+1} \Psi_0\big( (\momentmap - \Unit) z^K \cc z^L \big)
	=
	0
	\end{equation*}
	holds for all $K, L \in \mathbb N_0$ with $\abs K = \abs L$.
	Now consider $f = \sum_{k=0}^d f_k \in \ker \Psi_\hbar$, with $f_k \in 
	\Polynomials^{k,k}(\CC^{1+n})$ the homogeneous components of $f$ and $d\in \NN_0$. 
	Define
	\begin{equation*}
	g \coloneqq \sum_{k=0}^d \hbar^{k-d} \Big(\falling[\Big]{\frac 1 \hbar - k}{d-k}\Big)^{-1} 
	\momentmap^{d-k} f_k \in 
	\Polynomials^{d,d} (\CC^{1+n}) \punkt
	\end{equation*} 
	It follows from Lemma~\ref{lemma:moduloSId1} that $g-f \in \genSId{\momentmap - 1}$,
	hence $g = (g-f) + f \in \ker \Psi_\hbar$. But on homogeneous polynomials, the maps 
	$\Psi_\hbar$ and $\Psi_0$ coincide up to a non-zero scalar factor, so $\Psi_0(g)=0$.
	Therefore $g = 0$ by homogeneity, and $f \in \genSId{\momentmap - 1}$ holds.
	This shows that $\ker \Psi_\hbar = \genSId{\momentmap-1}$.
	
	Surjectivity of $\Psi_\hbar$ is clear because since $\Psi_\hbar$ and $\Psi_0$ coincide up to a non-zero scalar factor on all homogeneous polynomials,
	and an easy calculation shows that $\Psi_\hbar(f^*)([w]) = \cc{\Psi_\hbar(f)([w])}$ for all $f\in \Polynomials_\hbar(\CC^{1+n})^{\lie u_1}$, $[w]\in\CC\PP^n$.
\end{proof}
As a consequence one can endow $\Polynomials(\CC\PP^n)$ with a reduced (rational) star product:
\begin{definition}
  For all $\hbar \in \RR \setminus \set{\frac1 k}{k\in \NN}$, define the $^*$\=/algebra
  $\Polynomials_\hbar(\CC\PP^n)$ as the vector space $\Polynomials(\CC\PP^n)$ with product
  $\star_{\red,\hbar} \colon \Polynomials(\CC\PP^n) \times \Polynomials(\CC\PP^n) \to \Polynomials(\CC\PP^n)$,
	\begin{equation}
    \big(\Psi_\hbar(f) , \Psi_\hbar(g) \big) \mapsto \Psi_\hbar(f) \star_{\red,\hbar} \Psi_\hbar(g) \coloneqq \Psi_\hbar(f \star_\hbar g)
	\end{equation}
  and with pointwise complex conjugation as $^*$\=/involution.
\end{definition}
Note that the only $\hbar \in \RR \setminus \set{\frac1 k}{k\in \NN}$ for which the $^*$\=/algebras $\Polynomials_\hbar(\CC\PP^n)$ and $\Polynomials_\hbar(\CC^{1+n})_{1\mred}$
are isomorphic is $\hbar = 0$.
In the following we will determine the $^*$\=/representations of $\Polynomials_\hbar(\CC\PP^n)$
for $\hbar \in \RR \setminus \big( \{0\} \cup \set{\frac{1}{k}}{k\in \NN}\big)$ by classifying all its quadratic modules.

If $\hbar \in {]0, \infty[} \setminus \set[\big]{\frac 1 k}{k \in \NN}$, then it follows immediately 
from Proposition~\ref{proposition:minusOne} that any $^*$\=/representation of $\Polynomials_\hbar(\CC\PP^n)$ must be trivial.
The following lemma can be used to transfer this to negative values of $\hbar$:

\begin{lemma} \label{lemma:signIsomorphism}
	For $\hbar \in {]0,\infty[}$ the map
	$\IsomSign{\hbar} \colon \Polynomials_\hbar(\CC^{1+n}) \to \Polynomials_{-\hbar}(\CC^{1+n})$,
	\begin{equation} \label{eq:isomDifferentSign}
	f
	\mapsto 
	\IsomSign{\hbar}(f)
	\coloneqq 
	\bigg(
	\exp\Big(
	{-}\hbar \sum\nolimits_{j=0}^n \frac{\partial^2}{\partial z_j \partial \cc z_j}
	\Big) f
	\bigg) \circ \overline{\argument} \komma
	\end{equation}
	where $\cc\argument \colon \CC^{1+n} \to \CC^{1+n}$ denotes the componentwise complex conjugation, is
	a $^*$\=/isomorphism, i.e.~$\IsomSign{\hbar}$ is a linear bijection that fulfils
	\begin{equation}
	  \IsomSign{\hbar}(f\star_\hbar g) = \IsomSign{\hbar}(f) \star_{-\hbar} \IsomSign{\hbar}(g)
	  \quad\quad\text{and}\quad\quad
	  \IsomSign{\hbar}(f^*) = \IsomSign{\hbar}(f)^*
	\end{equation}
  for all $f,g\in \Polynomials(\CC^{1+n})$.
\end{lemma}

\begin{proof}
  The map
  $\Polynomials(\CC^{1+n}) \ni f \mapsto 
  \vartheta_\hbar(f) \coloneqq \exp\big({-\hbar} \sum_{j=0}^n \frac{\partial^2}{\partial z_j 
  \partial \cc z_j}\big) f
  \in\Polynomials(\CC^{1+n})$
  fulfils $\IsomSign{\hbar}(f) = \vartheta_\hbar(f) \circ \cc{\argument}$ for all $f\in \Polynomials(\CC^{1+n})$.
  Now define the ``anti-Wick product''
  \begin{equation*}
    f \mathbin{\tilde{\star}_\hbar} g
    \coloneqq
    \sum_{K\in \NN_0^{1+n}} \frac{(-\hbar)^{\abs{K}}}{K!} \frac{\partial^{\abs{K}} f}{\partial z^K} \frac{\partial^{\abs{K}} g}{\partial \cc{z}^K}
    \in
    \Polynomials(\CC^{1+n})
  \end{equation*}
  for all $f,g\in \Polynomials(\CC^{1+n})$. It is well-known that the ``equivalence transformation'' $\vartheta_\hbar$ fulfils
  $\vartheta_\hbar(f\star_\hbar g) = \vartheta_\hbar(f) \mathbin{\tilde{\star}_\hbar} \vartheta_\hbar(g)$
  for all $f,g\in \Polynomials(\CC^{1+n})$, see 
  e.g.~\cite[Prop.~2.18]{waldmann:ANuclearWeylAlgebra}. This in combination with the identity
  $\frac{\partial}{\partial z_j}(f\circ \cc{\argument}) = (\frac{\partial f}{\partial \cc{z}_j}) 
  \circ \cc{\argument}$
  for $f\in \Polynomials(\CC^{1+n})$, $j\in \{0,\dots,n\}$ yields
  $\IsomSign{\hbar}(f\star_\hbar g) = \IsomSign{\hbar}(f) \star_{-\hbar} \IsomSign{\hbar}(g)$
  for all $f,g\in \Polynomials(\CC^{1+n})$. Checking that
  $\IsomSign{\hbar}(f^*) = \IsomSign{\hbar}(f)^*$ for all $f\in \Polynomials(\CC^{1+n})$ is straightforward.
\end{proof}

\begin{proposition}
	If 
	$\hbar \in \RR \setminus \big(
		\{0\} 
		\cup 
		\set[\big]{\frac 1 k}{k \in \NN} 
		\cup 
		\set[\big]{-\frac 1 {1+n+k}}{k \in \NN_0}
	\big)$,
	then $-\Unit \in \Polynomials_\hbar(\CC\PP^n)^{++}_\Hermitian$, so every $^*$\=/representation of $\Polynomials_\hbar(\CC\PP^n)$ 
	on a pre-Hilbert space is trivial, i.e.\ the zero representation.
\end{proposition}

\begin{proof}
	If $\hbar \in {]0, \infty[} \setminus \set[\big]{\frac 1 k}{k \in \NN}$,
	then we have seen in Proposition~\ref{proposition:minusOne} that $-\Unit \in 
	\Polynomials_\hbar(\CC^{1+n})^{\lie u_1}$ is, up to a contribution of 
	$(\genSId{\momentmap-1})_\Hermitian$, a sum of Hermitian squares.
	So $\Psi_\hbar(-\Unit) = -\Unit \in \Polynomials_\hbar(\CC\PP^n)^{++}_\Hermitian$.
	
	If $\hbar \in {]{-\infty}, 0[} \setminus \set{-\frac 1 {1+n+k}}{k \in \NN_0}$, then 
	$\mu \coloneqq 1 - \abs \hbar(1+n)$ and $\abs \hbar$ fulfil the 
	assumptions of 
	Proposition~\ref{proposition:minusOne}.
	Therefore $-\Unit \in (\Polynomials_{\abs \hbar}(\CC^{1+n})^{\lie u_1})^{++}_\Hermitian 
	+ 
	(\genSId{\momentmap - \mu })_\Hermitian$ holds. Applying the $^*$\=/isomorphism $\IsomSign{\abs{\hbar}}$ from 
	the previous Lemma~\ref{lemma:signIsomorphism} and noting that 
	$\IsomSign{\abs{\hbar}}(\momentmap) = \momentmap - \abs{\hbar} (1+n) \Unit$,
	we find that 
	$-\Unit \in (\Polynomials_{\hbar}(\CC^{1+n})^{\lie u_1})^{++}_\Hermitian 
	+ 
	(\genSId{\momentmap - 1})_\Hermitian$. So again $\Psi_\hbar(-\Unit) = -\Unit \in \Polynomials_\hbar(\CC\PP^n)^{++}_\Hermitian$.
\end{proof}

\begin{proposition} If $\hbar = -\frac 1 {1+n+k}$, $k \in \NN_0$, then
  there exists only one quadratic module $\Q_\hbar$ on $\Polynomials_\hbar(\CC\PP^n)$
  with $-\Unit \notin \Q_\hbar$, and $\Polynomials_\hbar(\CC\PP^n) / \supp_\CC \Q_\hbar$
  is a finite dimensional C$^*$\=/algebra isomorphic to the matrix $^*$\=/algebra 
  $\CC^{d_{n,k} \times d_{n,k}}$ with $d_{n,k} = \binom{n+k}{k}$.
\end{proposition}

\begin{proof}
  Consider the map $\Psi_\hbar \circ \IsomSign{\abs \hbar} \colon \Polynomials_{\abs{\hbar}}(\CC^{1+n})^{\lie u_1} \to \Polynomials_\hbar(\CC\PP^n)$.
  By Lemmas~\ref{lemma:cpnProduct} and \ref{lemma:signIsomorphism}, $\Psi_\hbar \circ \IsomSign{\abs \hbar}$
  is a surjective $^*$\=/homomorphism (i.e.~linear, multiplicative, and intertwines the $^*$\=/involutions).
  Its kernel is the $^*$\=/ideal of $\Polynomials_{\abs {\hbar}}(\CC^{1+n})^{\lie u_1}$ that is generated by $\momentmap-\abs{\hbar} k\Unit$,
  because $\IsomSign{\abs \hbar}(\momentmap-\abs{\hbar} k \Unit) = \momentmap - \abs{\hbar}(1+n + k) \Unit = \momentmap-\Unit$.
  Because of this, there is a $1$-to-$1$ correspondence between quadratic modules of $\Polynomials_\hbar(\CC\PP^n)$
  and quadratic modules of $\Polynomials_{\abs{\hbar}}(\CC^{1+n})^{\lie u_1}$ that contain $(\genSId{\momentmap-\abs{\hbar}k})_\Hermitian$.
  
  It only remains to show that the quadratic module $\R_{\abs{\hbar}, \abs{\hbar} k}$ of $\Polynomials_{\abs{\hbar}}(\CC^{1+n})^{\lie u_1}$
  is the unique one that contains $(\genSId{\momentmap-\abs{\hbar}k})_\Hermitian$, but not $-\Unit$:
	As a consequence of Theorem~\ref{theorem:main}, every quadratic module of
	$\Polynomials_{\abs{\hbar}}(\CC^{1+n})^{\lie u_1}$ containing $(\genSId{\momentmap - \abs{\hbar}k})_\Hermitian$
	must contain $\R_{\abs{\hbar}, \abs{\hbar} k}$. But $\R_{\abs{\hbar}, \abs{\hbar} k}$ is also maximal
	under all quadratic modules of $\Polynomials_{\abs{\hbar}}(\CC^{1+n})^{\lie u_1}$ that do not contain $-\Unit$:
	This follows from the observation that its image under the reduction map $[\argument]_{\abs{\hbar}k} \colon \Polynomials_{\abs{\hbar}}(\CC^{1+n})^{\lie u_1} \to \Polynomials_{\abs{\hbar}}(\CC^{1+n})_{\abs{\hbar}k\mred} \cong \CC^{d_{n,k} \times d_{n,k}}$ is given by the positive-semidefinite Hermitian matrices
	in $\CC^{d_{n,k} \times d_{n,k}}$, which is a maximal quadratic module, see e.g.~\cite[Sec.~2]{cimpric:maximalQuadraticModulesOnStarRings}.
\end{proof}


\end{onehalfspace}

\end{document}